\documentclass[review]{elsarticle}
\usepackage{}

\usepackage[latin1]{inputenc} % arXiv±ØÐë¼ÓÕâ¸ö£¬·ñÔòÍ¶¸å±àÒë²»³É¹¦
\usepackage{listings, tcolorbox}% arXiv±ØÐë¼ÓÕâ¸ö£¬·ñÔòÍ¶¸å±àÒë²»³É¹¦

\usepackage{lineno,hyperref}
\usepackage{amssymb}
\usepackage{color}
\usepackage{amsmath}
\usepackage{tikz}
\usepackage{amsfonts}
\usepackage{mathrsfs}
\usepackage{amsthm}

\usepackage{bbding}
\usepackage{mathrsfs}
\usepackage{amsmath, amsfonts, amssymb, mathrsfs, txfonts}
\usepackage{graphicx, subfigure}
\usepackage{wasysym}
\usepackage{color}
\usepackage{bm}
\usepackage{enumerate}

\usepackage{pifont}
\usepackage{txfonts}
\usepackage{amsmath}
\usepackage{graphicx}
\usepackage{amsfonts}
\usepackage{amssymb}
\usepackage{mathrsfs,psfrag,eepic,epsfig}
\usepackage{epstopdf}

\DeclareMathOperator*{\res}{Res}

\modulolinenumbers[5]

\journal{Journal of \LaTeX\ Templates}

\makeatletter \@addtoreset{equation}{section}

\newtheorem{prop}{Proposition}[section]
\newtheorem{thm}[prop]{Theorem}
\newtheorem{cor}[prop]{Corollary}
\newtheorem{lem}[prop]{Lemma}
\theoremstyle{definition}

\newtheorem{rem}[prop]{Remark}
\newtheorem{assum}[prop]{Assumption}
\newtheorem{RHP}[prop]{RH Problem}

\renewcommand{\baselinestretch}{1.25}
\begin{document}

\begin{frontmatter}

\title{On the asymptotic stability of $N$-soliton solutions of the modified nonlinear Schr\"{o}dinger equation\tnoteref{mytitlenote}}
\tnotetext[mytitlenote]{%Project supported by the Fundamental Research Fund for the Central Universities under the grant No. 2017XKQY101.\\
%\hspace*{3ex}$^{*}$
Corresponding author.\\
\hspace*{3ex}\emph{E-mail addresses}: sftian@cumt.edu.cn,
shoufu2006@126.com (S. F. Tian) }

%% Group authors per affiliation:
\author{Jin-Jie Yang, Shou-Fu Tian$^{*}$ and Zhi-Qiang Li}
\address{School of Mathematics, China University of Mining and Technology,  Xuzhou 221116, People's Republic of China}

\begin{abstract}\quad
The Cauchy problem of the modified nonlinear Schr\"{o}dinger (mNLS) equation with the finite density type initial data is investigated via $\overline{\partial}$ steepest descent method. In the soliton region of space-time $x/t\in(5,7)$, the long-time asymptotic behavior of the mNLS equation is derived for large times. Furthermore,
for general initial data in a non-vanishing background, the soliton resolution conjecture for the mNLS equation is verified, which means that the asymptotic expansion of the solution can be characterized by finite number of soliton solutions as the time $t$ tends to infinity, and a residual error $\mathcal {O}(t^{-3/4})$ is provided.
\end{abstract}


\end{frontmatter}

%\linenumbers
\tableofcontents
%\newpage

\section{Introduction}\quad
Nonlinear Schr\"{o}dinger (NLS) equation is a classical integrable  model with Lax pairs, infinite conservation laws and Hamiltonian structures
\begin{align}\label{e1}
iq_t+q_{xx}+2|q|^2q=0,
\end{align}
which can be used to describe nonlinear optics, plasma and other phenomena. However, in some physical applications, such as the propagation of femtosecond optical pulses, it is necessary to add some higher-order effects to the NLS equation. In addition, for describing the resonance phenomenon, it is necessary to consider the multi-component formal equation.  This has led to a general interest in the study of higher order integrable models.

A completely integrable system, the modified nonlinear Schr\"{o}dinger equation \cite{Wadati-JPSJ},  was firstly proposed by Wadati et al. in 1970s
\begin{align}\label{e2}
q_t-iq_{xx}+\gamma\left(|q|^2q\right)_x+i\beta|q|^2q=0,
\end{align}
where $\alpha$ and $\beta$ are constants. This integrable model is of great significance since it can be reduced to different integrable models for taking different values about $\gamma$ and $\beta$ as follows.
\begin{enumerate}[(i)]
\item For $\gamma=0$, the model \eqref{e2} reduces to the classical NLS equation \eqref{e1}.
\item For $\beta=0$, the Kaup-Newell equation \cite{Kundu-1} can be derived from the model \eqref{e2}.
\item For $\gamma,\beta\neq0$, the model is equivalent to the modified Schr\"{o}dinger (mNLS) equation (or the perturbation NLS equation \cite{Maimistov-1993})
\begin{align}\label{q1}
iq_{t}+q_{xx}+2|q|^{2}q+i\frac{1}{\alpha}\left(|q|^{2}q\right)_{x}=0,
\end{align}
where $\alpha$ is a positive constant.
\end{enumerate}
Different from the classical NLS equation, the mNLS equation has static local solutions when the effective nonlinear parameters satisfy special conditions, which was shown in  \cite{Brizhik-2003}. The mNLS equation \eqref{q1} has many applications in physics, such as describing gravitational waves in deep water \cite{Stiassnie-1984}, Alfv\'{e}n waves propagating along the magnetic field in cold plasma \cite{Mio-1976}, and short pulses propagating in long optical fibers with nonlinear refractive index \cite{Nakatsuka-1981,Tzoar-1981}. Additionally, the nonlinear term $i\left(|q|^{2}q\right)_{x}$ appearing in \eqref{q1} is called self-steepening effect term, which leads to the propagation of optical pulse asymmetric and steepens upward at the trailing edge \cite{Agrawal-2007,Yang-2010}. Therefore, the mNLS equation has been widely concerned. By using analytic methods, Hirota method and Darboux transformation, different types of solutions were constructed, including soliton solutions, rational solutions and rouge wave solutions \cite{Mihalache-1993,Liu-1993,He-2015,Wen-2015}. For the initial value in Schwartz space, the $N$-soliton solutions of the mNLS equation \eqref{q1} were obtained by the inverse scattering method and the dressing method \cite{Chen-1990,Chen-1991,Doktorov-2007}. Additionally, the unique existence, existence of global solution, weak solution and other properties of solutions to the mNLS equation \eqref{q1}  with initial value problem were studied in \cite{Guo-1991,Tian-1994,Tian-2004}.

Recently, Fan's team studied the soliton solutions of equation \eqref{q1} under non-zero boundary conditions (NZBCs) by using the Riemann-Hilbert (RH) problem \cite{Yang-PD-2021}, and studied the long-time asymptotic behavior of the equation under the condition that the initial value belongs to Schwarz space \cite{Chen-JMAA} by using the nonlinear steepest descent method proposed by Defit and Zhou (DZ) in \cite{DZ-1993}. Our purpose in the present work is to study the asymptotic stability of $N$-soliton solution of the mNLS equation \eqref{q1} under the condition of NZBCs by using the classical DZ method and the Dbar method, i.e.,  under the NZBCs
\begin{align}\label{q2}
q(x,t)=q_{\pm}e^{-4i\alpha^{2}t+2i\alpha x},\quad x\rightarrow\pm\infty,
\end{align}
where $q_{\pm}$ are independent of variables $x$ and $t$ and $|q_{\pm}|=1$, we will  verify that the soliton resolution conjecture holds for the initial value $q(x,0)=q_0$ satisfying the conditions $q_0-q_{\pm}\in H^{1,1}(\mathbb{R})$.

Actually, the soliton resolution conjecture is a vague statement, but it is widely believed that for many globally well posed nonlinear dispersion equations, the time evolution of general initial data can be decomposed into finite soliton sequences with different velocities plus the dispersive radiation component in the large time limit. For the fact that the linear Schr\"{o}dinger equation is dispersive, the nonlinear term can balance the dispersive term and produce  stable soliton solutions when the nonlinear term is considered. For the NLS equation \eqref{e1}, there are two cases, focusing and defocusing, which correspond to the non self-adjoint operator and self-adjoint operator respectively. With the fast decay initial value, the focusing NLS equation has bright soliton solution, while the defocusing NLS equation has no soliton solution at this time.

For the focusing NLS equation with the initial data belongs to Schwarz space, McLaughlin and Miller in \cite{McLaughlin-1,McLaughlin-2}
proposed a new asymptotic analysis method, the so-called Dbar steepest descent method, combining Dbar problem \cite{Beals-Coifman,Book,ZJY-JPA,FEG-AML} with  Riemann-Hilbert (RH) problem \cite{Ablowitz-1991,GBL-JMAA,RHP-1,RHP-2} to verify that the soliton decomposition conjecture is true, and a cone was found, which can be used to approximate $N$-soliton solution by using solitons appearing on the cone in \cite{fNLS}. This method is also successfully applied to the defocusing NLS equation \cite{McLaughlin-3,SandRNLS}. It is worth noting that in \cite{SandRNLS}, Cuccagna and Jenkins studied the the defocusing NLS equation with initial value of finite density type, in which it has dark soliton solutions, and exhibited  the asymptotic stability results of $N$-dark soliton solutions. Then this method is used to study the soliton decomposition of many integrable models, including the modified KdV equation \cite{LJQ-1}, derivative NLS equation \cite{LJQ-2,Jenkins-CMP,Jenkins-CPDE,Y-DNLS}, the short pulses equation \cite{Fan-1,Fan-2}, Fokas-Lenells equation \cite{Fan-3}, the nonlocal NLS equation \cite{NNLS}, the modified Camassa-Holm equation \cite{Fan-4}, Hirota equation \cite{Yang-Hirota} and other equations. Certainly, the long-time asymptotic behavior of NLS equation has also been studied by using the classical nonlinear descent (DZ) method. The advantage of the Dbar steepest descent method is that it not only improves the error accuracy, but also avoids the $L^p$-estimation involving Cauchy projection operator.
Specifically, via the inverse scattering method, the long-time asymptotic behavior of the defocusing NLS equation without the error term $\varepsilon(x,t)$ was firstly established by Zakharov \cite{Zakharov-1976} in 1976. For sufficiently smooth initial value, the error term $\varepsilon(x,t)=\mathcal {O}(t^{-1}\log t)$ was shown in  \cite{Deift-1994,DZ-1994}. Subsequently, for the initial value belongs to the weighted Sobolev space $H^{1,1}(\mathbb{R})$, the error term was improved to  $\varepsilon(x,t)=\mathcal {O}(t^{-(\frac{1}{2}+\iota)})$ for any $0<\iota<\frac{1}{4}$ in \cite{DZ-Sobolev}. Additionally, a series of excellent results have been obtained by using the classical DZ method, including the focusing NLS equation \cite{NLS-CMP-BA,NLS-CMP-BG}, the short pulse equation \cite{Xu-JDE1,Xu-JDE2}, Gerdjikov-Ivanov equation \cite{Tian-PAMS} and other equations \cite{Geng-CMP,Liu-CMS}.

Since the RH method is used to obtain the soliton solution of the mNLS equation \eqref{q1} under the NZBCs \eqref{q2} shown in \cite{Yang-PD-2021}, and the classical DZ method is also used to obtain the long-time asymptotic behavior of the initial value belonging to the Schwarz space \cite{Chen-JMAA}. Hereafter the main purpose of this work is to study the long-time asymptotic behavior of the solution under the NZBCs \eqref{q2}, and analyze the asymptotic stability of the soliton solution by  Dbar steepest descent method.  Compared with the known results, this work has the following improvements.
\begin{enumerate}[(I)]
\item The initial data $q$ belongs to the Schwarz space in \cite{Chen-JMAA}, then the asymptotic solution is of the form
\begin{align*}
q(x,t)=t^{-1/2}\alpha(z_{0})e^{\frac{i\pi^2}{4t}-i\nu(z_{0})\log8t}e^{-\frac{i\alpha}{2\pi}\int_{-\infty}^{\lambda_0}
\log(1-|r(\lambda)|^2)(a\lambda-b)d\lambda}+\mathcal {O}(t^{-1}\log t).
\end{align*}
In our work, we consider the NZBCs \eqref{q2}, and the error is improved to  $\mathcal {O}(t^{-3/4}).$
\item The Dbar steepest descent method is used in this work, instead of the classical DZ steepest descent method.
The main difference between the two methods is that the Dbar steepest descent method does a continuous extension of the jump matrix after the triangular decomposition, while the classical DZ method is an analytic extension.
\item  The existence of solitons is also considered, that is, the existence of discrete spectrum is considered in the original RH problem, which makes it difficult to establish long-time asymptotic solutions for the mNLS equation  under NZBCs.
\item  Compared with the asymptotic solutions in partial differential equations, we verify the soliton resolution conjecture of the modified NLS equation, and give its exact expression, including the exact soliton solution and error term $\mathcal {O}(t^{-3/4}).$
\end{enumerate}

Some notations related to this work are given as follows:
\begin{enumerate}[(i)]
\item  The Sobolev space
\begin{align*}
W^{j,k}(\mathbb{R})=\left\{f(x)\in L^k(\mathbb{R}):\partial_x^mf(x)\in L^k(\mathbb{R}),~m=1,\ldots,j~\right\}.
\end{align*}
\item  The weighted Sobolev space
\begin{align*}
H^{j,k}(\mathbb{R})=\left\{f(x)\in L^2(\mathbb{R}):\partial_x^mf(x)\in L^{2}(\mathbb{R}), |x|^kf(x)\in L^{2}(\mathbb{R}), ~m=1,\ldots,j~\right\}.
\end{align*}
\item  The weighted $L^p(\mathbb{R})$ space
\begin{align*}
L^{p,s}(\mathbb{R})=\left\{f(x)\in L^p(\mathbb{R}):|x|^sf(x)\in L^p(\mathbb{R})\right\}.
\end{align*}
\item  Define $A\lesssim B$ for any quantities $A$ and $B$, if there exists a constant $c>0$ such that $|A|\leq c|B|$.
Additionally, the norm of $f(x)\in L^{p}(\mathbb{R})$ $(L^{p,s}(\mathbb{R}))$ are abbreviated to $||f||_p$ $(||f||_{p,s})$, respectively.
\end{enumerate}

\noindent {\textbf{Structure of the work}}\\

The theorem \ref{first} is proved by employing the Dbar  technique and the Riemann-Hilbert problem to the mNLS equation \eqref{q1}.

In Sec. 2 and 3, we mainly review the spectral analysis process of the mNLS equation \eqref{q1}, including the selection of non-zero boundary value, the establishment of Jost functions, and time evolution of the scattering data, which are the basis for long-time asymptotic analysis.

In Sec. 4, based on the properties of the modified eigenfunctions $\mu_{\pm}(x,t;z)$ and scattering coefficients $s_{11}(z)$ and $s_{22}(z)$ analyzed in Sec. 2 and 3, the RH problem \ref{RHP1} corresponding to the mNLS equation \eqref{q1} with NZBCs
\eqref{q2} is established, and the solution of the RH problem is connected with the solution of the mNLS equation. According to the RH problem \ref{RHP1}, the long-time asymptotic analysis is performed in Sec. 5. The process is similar to the classical DZ method, that is, the original RH problem $M^{(0)}(z)$ is transformed by equivalent deformation or controllable error deformation. The specific implementation process can be described as $M^{(0)}(z)\leftrightarrows M^{(1)}(z)\leftrightarrows M^{(2)}(z)\leftrightarrows M^{(3)}(z) \leftrightarrows M^{(err)}(z)$.

In Sec. 5.1, we analyze the sign distribution shown in Fig. 3 of the oscillation terms $e^{\pm2it\theta}$ appearing in the jump matrix $V(z)$ \eqref{jumpv} in the original RH problem \ref{RHP1}, which is the core of the triangular decomposition of jump matrix. Additionally, the first deformation of the original RH problem is carried out, and a new unknown function $M^{(1)}(z)$ satisfying the RH problem \ref{RHP2} is obtained, which not only trades the residue conditions of poles to the jump on some disks, but also satisfies the core step of nonlinear descent method, that is, the jump on the real axis $\mathbb{R}$ of the original RH problem is decomposed into a new contour by continuous extension, so that the jump matrix on the new contour is bounded as $t\rightarrow\infty$. In Sec. 5.2, a matrix-valued function $R^{(2)}(z)$ defined in \eqref{R2} with controllable errors is introduced to transform $M^{(1)}(z)$ into $M^{(2)}(z)$ satisfying the mixed RH problem \ref{RHP3}, which includes soliton component and pure Dbar problem.

In Sec. 5.3, we consider $M^{(sol)}(z)$ satisfied the RH problem \ref{RHP4} corresponding to zero derivative of Dbar in RH problem \ref{RHP3}. Since the jump matrix $V^{(2)}(z)$ defined by \eqref{jumpv2} in $M^{(2)}(z)$ converges uniformly to the identity matrix, we further decompose $M^{(2)}(z)$ into $M_\Lambda^{(sol)}(z)$ defined by \eqref{ML}, where the error function $M^{(err)}(z)$ satisfies the theory of small norm RH problem \ref{RHP6} shown in  Sec. 5.4. It is proved that the solution of $M^{(sol)}(z)$ is equivalent to the solution of the original RH problem \ref{RHP1} under the condition of modified scattering data $\left\lbrace  0,\left\lbrace \zeta_n,\mathring{c}_n\right\rbrace_{n\in\Lambda}\right\rbrace$ presented in proposition \ref{pp-deng}, which provides a leading term for the long-time asymptotic behavior of the solution. In Sec. 5.5,  the pure Dbar problem \ref{RHP7} satisfied by $M^{(3)}(z)$ is discussed, which provides error terms.
Summarizing the above analysis gives the proof of theorem \ref{first} in Sec. 5.6.
\section{Jost functions}\quad
In this section, we briefly describe the direct scattering problem of the modified nonlinear Schr\"{o}dinger equation \eqref{q1} with finite density type initial data. These results are known, and interested readers can refer to the reference \cite{Yang-PD-2021}, which has a detailed proof process.

The modified nonlinear Schr\"{o}dinger equation is an integrable model with the following Lax representation
\begin{align}\label{Lax1}
\phi_{x}=U\phi,\quad \phi_{t}=V\phi,
\end{align}
where $U=U(x,t;z)$ and $V=V(x,t;z)$ are $2\times2$ matrices defined by
\begin{align}
&U=-\alpha i (k^2-1)\sigma_3+ikQ,\\
&V=-2i\alpha^2(k^2-1)^2\sigma_3+2i\alpha k(k^2-1)Q+ik^2Q^2\sigma_3-k\sigma_3Q_x-\frac{i}{\alpha}kQ^3,
\end{align}
with
\begin{equation}
\hspace{0.5cm}\sigma_3=\left(\begin{array}{cc}
1 & 0   \\
0 & -1
\end{array}\right),\hspace{0.5cm}Q=\left(\begin{array}{cc}
0 & q  \\
\overline{q} &0
\end{array}\right).\nonumber
\end{equation}

Obviously, the constant boundary value can not be obtained directly from mNLS equation \eqref{q1}. For the convenience of later processing, the following changes are introduced
\begin{align*}
&q \rightarrow qe^{-4i\alpha^2t+2i\alpha x},\\
&\phi \rightarrow e^{(-2i\alpha^2t+i\alpha x)\sigma_3}\phi,
\end{align*}
then the  mNLS equation \eqref{q1}   becomes
\begin{align}\label{MNLS1}
iq_t+u_{xx}+4i\alpha q_x+i\frac{1}{\alpha} (|q|^2 q)_x=0,
\end{align}
with corresponding nonzero boundary conditions
\begin{align}\label{boundary}
\lim_{x \to \pm\infty}q(x,t)=q_\pm,
\end{align}
where $q_\pm$ are constant independent of $x, t$, and  $|q_\pm|=1$. Additionally, the corresponding Lax pair  representation \eqref{Lax1} is transformed into
\begin{align}\label{Lax2}
\phi_x=X\phi,\hspace{0.5cm}\phi_t=T\phi,
\end{align}
where
\begin{align*}
X=-i\alpha k^2\sigma_3+ikQ,\hspace{0.5cm}Q=\left(\begin{array}{cc}
0 & q \\
\overline{q}  & 0
\end{array}\right),
\end{align*}
\begin{align*}
T=(-2i\alpha^2 k^4+4i\alpha^2 k^2)\sigma_3+ik^2|q|^2\sigma_3+2i\alpha k(k^2-2)Q-\frac{i}{\alpha}kQ^3+kQ_x.
\end{align*}

Fix the potential $q=q(x,t)$ such that the step-like initial data \eqref{boundary}. The spectrum problem \eqref{Lax2} can be written as an eigenvalue equation
\begin{align}\label{q3}
i\sigma_{3}\phi_{x}-\alpha k^{2}\phi+k\sigma_3 Q\phi=0.
\end{align}
Let
\begin{align}\label{eigenmatrix}
Y_{\pm}=Y_{\pm}(z)=\left(\begin{array}{cc}
1 & \frac{q_\pm}{z} \\
-\frac{\overline{q_\pm}}{z} & 1
\end{array}\right) =I+\frac{1}{z}\sigma_3Q_\pm,
\end{align}
and
\begin{align}\label{Jost}
\phi_{\pm}=Y_{\pm}e^{-ik(z)\lambda(z)x\sigma_3},
\end{align}
where
\begin{align}
Q_\pm=\left(\begin{array}{cc}
0 & q_\pm \\
q_\pm^* & 0
\end{array}\right),\quad k(z)=\frac{1}{2\alpha}(z-\frac{1}{z}),\hspace{0.5cm}\lambda(z)=\frac{1}{2}(z+\frac{1}{z}).
\end{align}
As usual, the functions $\phi_{\pm}$ are the solutions of the spectrum problem \eqref{q3} by replacing $Q$ with $Q_{\pm}$, which are called the Jost functions. To eliminate the oscillation term, the modified eigenfunctions $\mu_{\pm}=\mu_{\pm}(x;z)$ are introduced
\begin{align}\label{q4}
\mu_{\pm}(x;z)=\phi_{\pm}(x;z)e^{ik(z)\lambda(z) x\sigma_3},
\end{align}
which together with \eqref{Jost} yields the asymptotic behavior
\begin{align}\label{mu}
\mu_{\pm}(x;z)\sim Y_{\pm}(z),\quad x\rightarrow\pm\infty.
\end{align}
Additionally, $\det\mu_{\pm}=\det\phi_{\pm}=\det Y_{\pm}=1+z^{-2}$. The $\mu_{-}$ and $\mu_{+}$ are the unique solutions of the Volterra integral equations
\begin{align}
&\mu_-(x;z)=Y_-+\int_{-\infty}^{x}Y_-e^{-ik\lambda(x-y){\sigma}_3}[Y_-^{-1}\Delta X_-\mu_-(y;z)]e^{ik\lambda(x-y){\sigma}_3}dy,\label{jost1}\\
&\mu_+(x;z)=Y_+-\int_{x}^{\infty}Y_+e^{-ik\lambda(x-y){\sigma}_3}[Y_+^{-1}\Delta X_+\mu_+(y;z)]e^{ik\lambda(x-y){\sigma}_3}dy,\label{jost2}
\end{align}
where $\Delta X_\pm=X-X_{\pm}=ik(Q-Q_{\pm})$.

As shown in reference \cite{Yang-PD-2021}, for the initial values $q-q_{\pm}\in L^{1}(\mathbb{R})$, the following theorem can be proved by constructing a Neumann series.
\begin{prop}\label{pp1}
For the initial data $q-q_{\pm}\in H^{1,1}(\mathbb{R})$ and $t\in \mathbb{R}^{+}$, the modified eigenfunctions $\mu_{-,1}$ and $\mu_{+,2}$ are analytic in the region $D^{-}$,  the modified eigenfunctions $\mu_{+,1}$ and $\mu_{-,2}$ are analytic in the region $D^{+}$, where $\mu_{\pm,i}$ ($i=1,2$) denote the $i$-th columns of the eigenfunctions $\mu$, and  $D^{\pm}$ are equipped with $D^-=\{z|ImzRez>0\}$  and $D^+=\{z|ImzRez<0\}$, respectively, shown in Figure 1.
\end{prop}

\centerline{
		\begin{tikzpicture}[node distance=2cm]
		\filldraw[gray!40,line width=2] (2.4,0.01) rectangle (0.01,2.4);
		\filldraw[gray!40,line width=2] (-2.4,-0.01) rectangle (-0.01,-2.4);
		\draw[->](-2.5,0)--(3,0)node[right]{$\mathbb{R}$};
		\draw[->](0,-2.5)--(0,3)node[above]{$i\mathbb{R}$};
		\draw[-](0,0)--(-0.8,0);
		\draw[-](-0.8,0)--(-1.8,0);
		\draw[-](0,0)--(0.8,0);
		\draw[-](0.8,0)--(1.8,0);
		\coordinate (A) at (0.5,1.2);
		\coordinate (B) at (0.6,-1.2);
		\coordinate (G) at (-0.6,1.2);
		\coordinate (H) at (-0.5,-1.2);
		\coordinate (I) at (0.16,0);
		\fill (A) circle (0pt) node[right] {$D^-$};
		\fill (B) circle (0pt) node[right] {$D_+$};
		\fill (G) circle (0pt) node[left] {$D_+$};
		\fill (H) circle (0pt) node[left] {$D^-$};
		\fill (I) circle (0pt) node[below] {$0$};
	\end{tikzpicture}}\label{fig:figure1}
\centerline{\noindent {\small \textbf{Figure 1.} The domains  $D^-$, $D^+$  and boundary  $\Sigma=\mathbb{R}\cup i\mathbb{R}\backslash \{0\}$.}}

\begin{lem}
Suppose that $q-q_{\pm}\in L^1(\mathbb{R})$ and that $q'\in W^{1,1}(\mathbb{R})$. Then the eigenfunctions $\mu_\pm(x;z)$ admit that
\begin{subequations}
\begin{align}
&\mu_\pm(x;z)=e^{i\nu_{\pm}(x,q)\sigma_{3}}+\mathcal {O}(z^{-1}),\quad z\rightarrow\infty, \label{q5}\\
&\mu_\pm(x;z)=\frac{1}{z}e^{i\nu_{\pm}(x,q)\sigma_{3}}\sigma_{3}Q_{\pm}+\mathcal {O}(1),\quad z\rightarrow0,\label{q6}
\end{align}
\end{subequations}
where $\nu_{\pm}(x,q)=\frac{1}{2\alpha}\int_{\pm\infty}^{x}\left(1-|q|^{2}\right)dy$.
\end{lem}

\section{Scattering data}\quad
In this section, we will rely on the Jost functions $\phi_{\pm}$ and the modified eigenfunctions $\mu_{\pm}$ generated by the initial value $q_{0}(x)$ to establish the corresponding scattering data $S(z)$, including the time evolution of the scattering data and the distribution of the discrete spectrum. Let's start with the following lemma.
\begin{lem}
Let $q-q_{\pm}\in L^{1}(\mathbb{R})$, then
\begin{enumerate}[(1)]
\item For $z\in\mathbb{C}\backslash\{-i,0,i\}$, the modified eigenfunctions
\begin{align}
\mu_{\pm}(x;z)=\left(\mu_{\pm,1}(x;z),~\mu_{\pm,2}(x;z)\right)=\left(\phi_{\pm,1}(x;z),~\phi_{\pm,2}(x;z)\right)
e^{ik\lambda x\sigma_3},
\end{align}
are nonsingular solutions of \eqref{q3} and
\begin{align}
\det\mu_{\pm}=\det\mu_{\pm}(z)=1+z^{-2}.
\end{align}
\item First symmetry
\begin{align}\label{mu-asy}
\mu_{\pm}(x;z)=\sigma_2\overline{\mu_{\pm}(x;\overline{z})}\sigma_2=\sigma_1\overline{\mu_{\pm}(x;-\overline{z})}\sigma_1.
\end{align}
\item Second symmetry
\begin{align}
\mu_{\pm}(x;z)=\frac{1}{z}\mu_{\pm}(x;-z^{-1})\sigma_3Q_{\pm}.
\end{align}
\end{enumerate}
\end{lem}

Since $\phi_{\pm}$ are the solutions of the eigenvalue problem \eqref{q3}, there exists a constant matrix $S(z)$ which is called scattering matrix satisfying
\begin{align}\label{SM}
\phi_+(x;z)=\phi_-(x;z)S(z),
\end{align}
for the individual columns
\begin{subequations}\label{scattering}
\begin{align}
&\phi_{+,1}(x;z)=s_{11}(z)\phi_{-,1}(x;z)+s_{21}(z)\phi_{-,2}(x;z),\\
&\phi_{+,2}(x;z)=s_{12}(z)\phi_{-,1}(x;z)+s_{22}(z)\phi_{-,2}(x;z).
\end{align}
\end{subequations}
Further, the reflection coefficients are introduced
\begin{align}\label{reflection}
\rho(z)=\frac{s_{21}(z)}{s_{11}(z)},\hspace{0.5cm}\widetilde{\rho}(z)=\frac{s_{12}(z)}{s_{22}(z)}.
\end{align}
The properties of scattering data $s_{11}(z)$, $s_{22}(z)$ and reflection coefficients are presented in the following lemma.
\begin{lem}\label{lem2}
Let $z\in\mathbb{C}\backslash\{-i,0,i\}$, the scattering data  and reflection coefficients defined in \eqref{scattering} and \eqref{reflection} generated by the initial $q-q_{\pm}\in L^{1}(\mathbb{R})$ such that
\begin{enumerate}[(1)]
\item The scattering data can be defined by Jost functions
\begin{align}
&s_{11}(z)=\frac{{\rm Wr}(\phi_{+,1},~\phi_{-,2})}{\gamma}=\frac{{\rm Wr}(\mu_{+,1},~\mu_{-,2})}{\gamma},\label{s11}\\
&s_{22}(z)=\frac{{\rm Wr}(\phi_{-,1},~\phi_{+,2})}{\gamma}=\frac{{\rm Wr}(\mu_{-,1},~\mu_{+,2})}{\gamma},\label{s22}
\end{align}
where $\gamma=1+z^{-2}$.
\item Symmetry properties satisfied by scattering matrix
\begin{align}\label{q7}
\begin{split}
&S(z)=\sigma_2\overline{S(\overline{z})}\sigma_2=\sigma_1\overline{S(-\overline{z})}\sigma_1,\\
&S(z)=(\sigma_3Q_{-})^{-1}S(-z^{-1})\sigma_3Q_{+},
\end{split}
\end{align}
where $\sigma_{1}$ and $\sigma_{2}$ are the Pauli matrices
\begin{align*}
\sigma_2=\left(
  \begin{array}{cc}
    0 & -i \\
    i & 0 \\
  \end{array}
\right),~~~\sigma_1=\left(
             \begin{array}{cc}
               0 & 1 \\
               1 & 0 \\
             \end{array}
           \right).
\end{align*}
\item Symmetry properties satisfied by reflection coefficients
\begin{align}\label{RS}
\rho(z)=\overline{\widetilde{\rho}(-\overline{z})}=-\overline{\widetilde{\rho}(\overline{z})}=\frac{q_-}{\overline{q_-}}
\tilde{\rho}
\left(-z^{-1}\right)=-\frac{\overline{q_-}}{q_-}\overline{\rho}\left(-z^{-1}\right).
\end{align}
\item Asymptotic properties of scattering matrix
\begin{align}
	&S(z)=e^{-i\nu_0\sigma_3}+\mathcal {O}(z^{-1}),\hspace{0.5cm}z \rightarrow \infty,\label{asympsca1}\\
	&S(z)={\rm diag}\left(\frac{q_-}{q_+},\frac{q_+}{q_-}\right)e^{i\nu_0\sigma_3}+\mathcal {O}(z),\hspace{0.5cm}z \rightarrow 0,\label{asympsca2}
	\end{align}
	where
	\begin{equation}
	\nu_0=\frac{1}{2\alpha}\int_{-\infty}^{+\infty}(1-|q|^2)dy.\label{nu0}
	\end{equation}

\end{enumerate}
\end{lem}
\begin{cor}
Suppose that $q-q_{\pm}\in L^{1}(\mathbb{R})$. Then the scattering data $s_{11}(z)$ defined by \eqref{s11} is analytic in $D^+$ as well as continuous in $\overline{D^+}\backslash\{iq_{0}\}$.  $s_{22}(z)$ defined by \eqref{s22}
is analytic in $D^-$ as well as continuous in $\overline{D^-}\backslash\{iq_{0}\}$. Additionally, although $s_{12}(z)$ and $s_{21}(z)$ are not analytic, but continuous to $\Sigma\backslash\{\pm iq_{0}\}$.
\end{cor}
\subsection{The discrete spectrum}\quad
As usual, the discrete spectrum corresponds to the soliton solution, so we need to give the distribution of the discrete spectrum, where the  discrete spectrum of the scattering problem is composed of all the values $z\in\mathbb{C}\backslash\Sigma$  that make the eigenfunctions exist in $L^{2}(\mathbb{R})$. Additionally, there is another possibility, namely the so-called spectral singularity \cite{Z-CPAM-1989}, which means that the zero point of the scattering data $s_{11}(z)$ exists on the real axis $\mathbb{}$ and the imaginary axis $i\mathbb{R}$, which will lead to the generation of convergence points. In order to avoid this situation, we make the following assumption.
\begin{assum}\label{initialdata}
Let the initial data $q-q_{\pm}\in L^{1,1}(\mathbb{R})$, the scattering data generated by the initial data satisfies that
\begin{enumerate}[(1)]
\item There is no zeros on $\Sigma$, i.e., $s_{11}(z_{n})=0$ $(z_{n}\notin\Sigma)$.
\item The scattering data $s_{11}(z)$ only has finite number of simple zeros, i.e., $s_{11}(z_{n})=0$ but $s_{11}'(z_{n})\neq0$.
\item The reflection coefficients $\rho(z)$ and $\widetilde{\rho}(z)$ belong to $W^{2,\infty}(\Sigma)\cap W^{1,2}(\Sigma)$.
\end{enumerate}
\end{assum}

Now suppose that the scattering data $s_{11}(z)$ has finite simple zeros on the  region $D^{+}\cap\{z\in\mathbb{C}:Imz>0,|z|>1\}$ and the circle $\{z=e^{i\varphi}: 0<\varphi<\frac{\pi}{2} \}$ denoted by $z_{1}$, $z_{2}$, $\ldots$, $z_{N_1}$ and $w_1,\ldots,w_{N_2}$, respectively. Recall the symmetry \eqref{q7}, then one has
\begin{align}
&s_{11}(\pm z_n)=0 \Longleftrightarrow \overline{s_{22}(\pm \bar{z}_n)}=0   \Longleftrightarrow   \overline{s_{22}\left(\pm \frac{1}{z_n}\right)}=0\nonumber\\
&  \Longleftrightarrow   s_{11}\left(\pm \frac{1}{\bar{z}_n}\right)=0, \hspace{0.5cm}n=1,\ldots,N_1,\nonumber
\end{align}
and on the circle
\begin{equation}
	s_{11}(\pm \zeta_m)=0\Longleftrightarrow \overline{s_{11}(\pm \overline{\zeta}_m)}=0, \hspace{0.5cm}m=1,\ldots,N_2, \nonumber
\end{equation}
which form the discrete spectrum set $\mathcal {Z}=\{\eta_{n},\overline{\eta}_{n}\}_{n=1}^{4N_1+2N_2}$ shown in Figure 2, where $\eta_n=z_n$, $\eta_{n+N_1}=-z_n$, $\eta_{n+2N_1}=\overline{z}_n^{-1}$ and $\eta_{n+3N_1}=-\overline{z}_n^{-1}$ for $n=1,\ldots,N_1$;  $\eta_{m+4N_1}=\zeta_m$ and $\eta_{m+4N_1+N_2}=-\zeta_m$ for $m=1,\ldots,N_2$.

\centerline{
	\begin{tikzpicture}[node distance=2cm]
		\filldraw[gray,line width=3] (3,0.01) rectangle (0.01,3);
		\filldraw[gray,line width=3] (-3,-0.01) rectangle (-0.01,-3);
		\draw[->](-3,0)--(3,0)node[right]{Re$z$};
		\draw[->](0,-3)--(0,3)node[above]{Im$z$};
		\draw[red] (2,0) arc (0:360:2);
		\draw[-](0,0)--(-1.5,0);
		\draw[-](-1.5,0)--(-2.8,0);
		\draw[-](0,0)--(1.5,0);
		\draw[-](1.5,0)--(2.8,0);
		\draw[-](0,2.7)--(0,2.2);
		\draw[-](0,1.6)--(0,0.8);
		\draw[-](0,-2.7)--(0,-2.2);
		\draw[-](0,-1.6)--(0,-0.8);
			\coordinate (A) at (2,3);
		\coordinate (B) at (2,-3);
		\coordinate (C) at (-0.5546996232,0.8320505887);
		\coordinate (D) at (-0.5546996232,-0.8320505887);
		\coordinate (E) at (0.5546996232,0.8320505887);
		\coordinate (F) at (0.5546996232,-0.8320505887);
		\coordinate (G) at (-2,3);
		\coordinate (H) at (-2,-3);
		\coordinate (J) at (1.7320508075688774,1);
		\coordinate (K) at (1.7320508075688774,-1);
		\coordinate (L) at (-1.7320508075688774,1);
		\coordinate (M) at (-1.7320508075688774,-1);
		\fill (A) circle (1pt) node[above right] {$z_n$};
		\fill (B) circle (1pt) node[right] {$\overline{z}_n$};
		\fill (C) circle (1pt) node[left] {$-\frac{1}{z_n}$};
		\fill (D) circle (1pt) node[left] {$-\frac{1}{\overline{z}_n}$};
		\fill (E) circle (1pt) node[right] {$\frac{1}{\overline{z}_n}$};
		\fill (F) circle (1pt) node[right] {$\frac{1}{z_n}$};
		\fill (G) circle (1pt) node[left] {$-\overline{z}_n$};
		\fill (H) circle (1pt) node[below left] {$-z_n$};
		\fill (J) circle (1pt) node[right] {$\zeta_m$};
		\fill (K) circle (1pt) node[right] {$\overline{\zeta}_m$};
		\fill (L) circle (1pt) node[left] {$-\overline{\zeta}_m$};
		\fill (M) circle (1pt) node[left] {$-\zeta_m$};
	\end{tikzpicture}}\label{fig:figure2}
\centerline{\noindent {\small \textbf{Figure 2.} Distribution of the discrete spectrum $\mathcal{Z}$. The red one is  unit circle.}}

Additionally, the trace formula shows the dependence between the reflection coefficient and the discrete spectrum, which can be expressed as the following form
\begin{align}\label{trace}
s_{11}(z)=\prod_{j=1}^{4N_1+2N_2}\frac{z-\zeta_j}{z-\overline{\zeta}_j}\exp\left\{-\frac{1}{2\pi i}\int_{\Sigma}\frac{\log (1-\rho(s)\widetilde{\rho}(s))}{s-z}ds \right\},
\end{align}
and plays an important role in satisfying the $M(x;z)$ of the original RH problem for a long time asymptotic analysis.

In order to obtain the soliton solution, we need to consider the residue condition at discrete spectral points. Similar to \cite{Yang-PD-2021}, denote $c_n=s_{21}(z_{n})/s_{11}'(z_{n})$ as a normalized constant, the residue conditions can be expressed as
\begin{subequations}
\begin{align}
\res_{z=z_n}\left[\frac{\mu_{+,1}(z)}{s_{11}(z)}\right]&=c_ne^{2ik(z_n)\lambda(z_n)x}\mu_{-,2}(z_n),\\
\res_{z=-z_n}\left[\frac{\mu_{+,1}(z)}{s_{11}(z)}\right]&=-c_ne^{2ik(z_n)\lambda(z_n)x}\sigma_3\mu_{-,2}(z_n),\\
\res_{z=\overline{z}_n^{-1}}\left[\frac{\mu_{+,1}(z)}{s_{11}(z)}\right]&=\frac{\overline{q}_{-}}{q_{-}}\overline{z}_n^{-2}
\overline{c}_ne^{2ik(-\overline{z}_n^{-1})\lambda(-\overline{z}_n^{-1})x}\sigma_3\mu_{-,2}(-\overline{z}_n^{-1}),\\
\res_{z=-\overline{z}_n^{-1}}\left[\frac{\mu_{+,1}(z)}{s_{11}(z)}\right]&=-\frac{\overline{q}_{-}}{q_{-}}\overline{z}_n^{-2}
\overline{c}_ne^{2ik(-\overline{z}_n^{-1})\lambda(-\overline{z}_n^{-1})x}\mu_{-,2}(-\overline{z}_n^{-1}),\\
\res_{z=\overline{z}_n}\left[\frac{\mu_{+,2}(z)}{s_{22}(z)}\right]&=-\overline{c}_ne^{-2ik(\overline{z}_n)
\lambda(\overline{z}_n)x}\mu_{-,1}(\overline{z}_n),\\
\res_{z=-\overline{z}_n}\left[\frac{\mu_{+,2}(z)}{s_{22}(z)}\right]&=-\overline{c}_ne^{-2ik(\overline{z}_n)
\lambda(\overline{z}_n)x}\sigma_3\mu_{-,1}(\overline{z}_n),\\
\res_{z=\overline{z}_n^{-1}}\left[\frac{\mu_{+,2}(z)}{s_{22}(z)}\right]&=\frac{q_{-}}{\overline{q}_{-}}z_n^{-2}
c_ne^{-2ik(-\overline{z}_n^{-1})\lambda(-\overline{z}_n^{-1})x}\sigma_3\mu_{-,1}(-\overline{z}_n^{-1}),\\
\res_{z=-\overline{z}_n^{-1}}\left[\frac{\mu_{+,2}(z)}{s_{22}(z)}\right]&=\frac{q_{-}}{\overline{q}_{-}}z_n^{-2}
c_ne^{-2ik(-\overline{z}_n^{-1})\lambda(-\overline{z}_n^{-1})x}\mu_{-,1}(-\overline{z}_n^{-1}).
\end{align}
\end{subequations}
For $m=1,2,\ldots,N_{2}$, denoting $c_{m+N_1}=s_{21,m+N_1}/s_{11}'(\zeta_{m})$, one has
\begin{subequations}
\begin{align}
\res_{z=\zeta_m}\left[\frac{\mu_{+,1}(z)}{s_{11}(z)}\right]&=c_{m+N_1}e^{2ik(\zeta_m)\lambda(\zeta_m)x}\mu_{-,2}(\zeta_m),\\
\res_{z=-\zeta_m}\left[\frac{\mu_{+,1}(z)}{s_{11}(z)}\right]&=-c_{m+N_1}e^{2ik(\zeta_m)\lambda(\zeta_m)x}\sigma_3\mu_{-,2}(\zeta_m),\\
\res_{z=\overline{\zeta}_m}\left[\frac{\mu_{+,2}(z)}{s_{22}(z)}\right]&=-\overline{c}_{m+N_1}e^{-2ik(\overline{\zeta}_m)
\lambda(\overline{\zeta}_m)x}\mu_{-,1}(\overline{\zeta}_m),\\
\res_{z=-\overline{\zeta}_m}\left[\frac{\mu_{+,2}(z)}{s_{22}(z)}\right]&=-\overline{c}_{m+N_1}e^{-2ik(\overline{\zeta}_m)
\lambda(\overline{\zeta}_m)x}\sigma_3\mu_{-,1}(\overline{\zeta}_m).
\end{align}
\end{subequations}
In order to facilitate the discussion, we introduce the notation $C_n$ to satisfy that
$$C_n=-C_{n+N_1}=c_n,~ C_{n+2N_1}=-C_{n+3N_1}=\frac{\bar{q}_-}{q_-}\bar{z}_n^{-2}\bar{c}_n,~n=1,\ldots,N_1$$
$$C_{m+4N_1}=C_{m+4N_1+N_2}=-c_{m+N_1},~m=1,\ldots,N_2,$$
 and the  set $\sigma_d=  \left\lbrace \zeta_n,C_n\right\rbrace^{4N_1+2N_2}_{n=1}  $
is called the scattering data.

\subsection{Time evolution of the scattering data}\quad
So far, we have only considered the spectral problem, that is, the spatial Lax pair. Notice that the inverse scattering theory has a great advantage, that is, the evolution of scattering data is linear and trivial as the potential function $q(x,t)$ evolves based on equation \eqref{q1}. For details, please refer to the references \cite{Deift-1994,Faddeev-1987}, here we give a brief description. By calculating the partial derivative of the scattering relation \eqref{SM} with respect to $t$ and combining the part of Lax with respect to the time evolution, the relation $\partial_tS(z)=-(2\alpha k^2-4\alpha-\alpha^{-1})[\sigma_3,S(z)]$ can be obtained. Now the time evolution of scattering data and reflection coefficient is further expressed as
\begin{subequations}
\begin{align}
&s_{11,t}(z;t) =0,\hspace{0.5cm}  s_{21,t}(z;t) =-(2\alpha k^2-4\alpha-\alpha^{-1})k\lambda s_{21}(z;t),\\
&C(\eta_n)\rightarrow C(t,\eta_n)=c(0,\eta_n)e^{-(2\alpha k^2-4\alpha-\alpha^{-1})k(\eta_n)\lambda(\eta_n) t},\\
&\rho(z)\rightarrow \rho(t,z)=\rho(0,z)e^{-(2\alpha k^2-4\alpha-\alpha^{-1})k\lambda t}.
\end{align}
\end{subequations}
The further evolution of the scattering data over time $t$ can be expressed as
\begin{equation*}
\left\lbrace  e^{-(2\alpha k^2-4\alpha-\alpha^{-1})k\lambda t}\rho(z),\left\lbrace \eta_n,e^{-(2\alpha k^2-4\alpha-\alpha^{-1})^2-1)k(\eta_n)\lambda(\eta_n) t}C_n\right\rbrace^{4N_1+2N_2}_{n=1}\right\rbrace,
\end{equation*}
where $\left\lbrace  \rho(z),\left\lbrace \eta_n,C_n\right\rbrace^{4N_1+2N_2}_{n=1}\right\rbrace$ are derived from the initial value $q(x, 0) =
q_0(x)$. Additionally, we take the  phase function
\begin{equation}\label{CT}
\theta(z)=-k(z)\lambda(z)\left[\xi-(2\alpha k^2-4\alpha-\alpha^{-1}) \right],
\end{equation}
where for brevity denote $\theta_n=\theta(\zeta_n)$ and $\xi=x/t$.

\section{Inverse scattering: set up of the RH problem}\quad
For $z\in\Sigma$, combining the related properties of the modified eigenfunctions $\mu_{\pm}$ and the scattering data $S(z)$ shown in proposition \ref{pp1} and lemma \ref{lem2}, we introduce the piecewise meromorphic function
\begin{align}
M(z)=M(x,t;z):=\left\{ \begin{array}{ll}
\left(  \frac{\mu_{+,1}(x,t;z)}{s_{11}(z)}, \mu_{-,2}(x,t;z)\right),   &\text{as } z\in D^+,\\[12pt]
\left( \mu_{-,1}(x,t;z),\frac{\mu_{+,2}(x,t;z)}{s_{22}(z)}\right)  , &\text{as }z\in D^-,\\
\end{array}\right.
\end{align}
which satisfies the  following RH problem.
\begin{RHP}\label{RHP1}
Find a $2\times2$ matrix function $M(z)$ admits
\begin{enumerate}[(1)]
\item Analyticity: the matrix $M(z)$ is analytic in $\mathbb{C}\backslash\{\Sigma\cup\mathcal {Z}\}$.
\item Asymptotic behavior:
\begin{align}
&M(z)=e^{i\nu_{-}(x,q)\sigma_{3}}+\mathcal {O}(z^{-1}),\quad z\rightarrow\infty, \\
&M(z)=\frac{1}{z}e^{i\nu_{-}(x,q)\sigma_{3}}\sigma_{3}Q_{-}+\mathcal {O}(1),\quad z\rightarrow0.
\end{align}
\item Jump condition: the limits $M_{\pm}(z)=\mathop{\lim}_{\Sigma\ni z'\rightarrow z}M(x,t;z')$, then
\begin{align}
M_{+}(z)=M_{-}(z)V(z),\quad z\in\Sigma,
\end{align}
where
\begin{align}\label{jumpv}
V(z)=\left(\begin{array}{cc}
1-\widetilde{\rho}(z)\rho(z) & -e^{2it\theta}\widetilde{\rho}(z)\\
e^{-2it\theta}\rho(z) & 1
\end{array}\right).
\end{align}
\item Symmetries:
$M(z)=\sigma_2\overline{M(\bar{z})}\sigma_2$=$\sigma_1\overline{M(-\bar{z})}\sigma_1=\frac{i}{z}M(-1/z)\sigma_3Q_-$.
\item Residue conditions: the matrix function $M(z)$ has simple poles at the set $\mathcal {Z}$.
\begin{align}
&\res_{z=\zeta_n}M(z)=\lim_{z\to \zeta_n}M(z)\left(\begin{array}{cc}
0 & 0\\
C_ne^{-2it\theta_n} & 0
\end{array}\right),\label{RES1}\\
&\res_{z=\overline{\zeta}_n}M(z)=\lim_{z\to \overline{\zeta}_n}M(z)\left(\begin{array}{cc}
0 & -\overline{C}_ne^{2it\overline{\theta}_n}\\
0 & 0
\end{array}\right)\label{RES2}.
\end{align}
\item The solution of the mNLS \eqref{q1} is defined by
\begin{align}\label{recover-q}
q(x,t)=\lim_{z\rightarrow\infty}e^{i\nu_{\pm}(x,t;q)\sigma_3}(zM(z))_{12}.
\end{align}
\end{enumerate}
\end{RHP}

\section{The long time analysis}\quad
Starting from this section, we will use the Defit-Zhou steepest descent method \cite{DZ-1993} combined with the Dbar technique to do a long time asymptotic analysis of the mNLS equation \eqref{q1}. On the one hand, the triangular decomposition is performed on the jump matrix $V(z)$ that appears in the RH problem \ref{RHP1} similar to \cite{DZ-1993}. On the other hand, through the analysis of the obtained mixed RH problem, combined with Dbar  technique, not only the soliton solution of the mNLS equation \eqref{q1} can be obtained, but also the long time asymptotic solution is derived. Additionally, compared with the classic DZ method, the error accuracy is improved, that is, the error is $\mathcal {O}(t^{-3/4})$ at this time.
\subsection{Conjugation}\quad
Before the triangular decomposition of the jump matrix $V(z)$, we have to consider the oscillating terms $e^{\pm2it\theta}$ that appears in the jump matrix $V(z)$. This is because when one of the oscillating term decays to zero for $t\rightarrow\infty$, the other must be unbounded. By employing the phase function $\theta$ defined in \eqref{CT}, then the real part of $2it\theta$ can be expressed as
\begin{align}\label{Re}
Re(2it\theta)=txy\left[(\xi-6)\left(1+\frac{1}{(x^2+y^2)^2}\right)+(x^2-y^2)\left(1+\frac{1}
{(x^2+y^2)^4}\right)\right],
\end{align}
where $z=x+iy$ ($x,y\in\mathbb{R}$) and the symbol distribution of the further function $Im\theta$ is shown in Figure 3.
\\

\centerline{
{\rotatebox{0}{\includegraphics[width=3.0cm,height=2.75cm,angle=0]{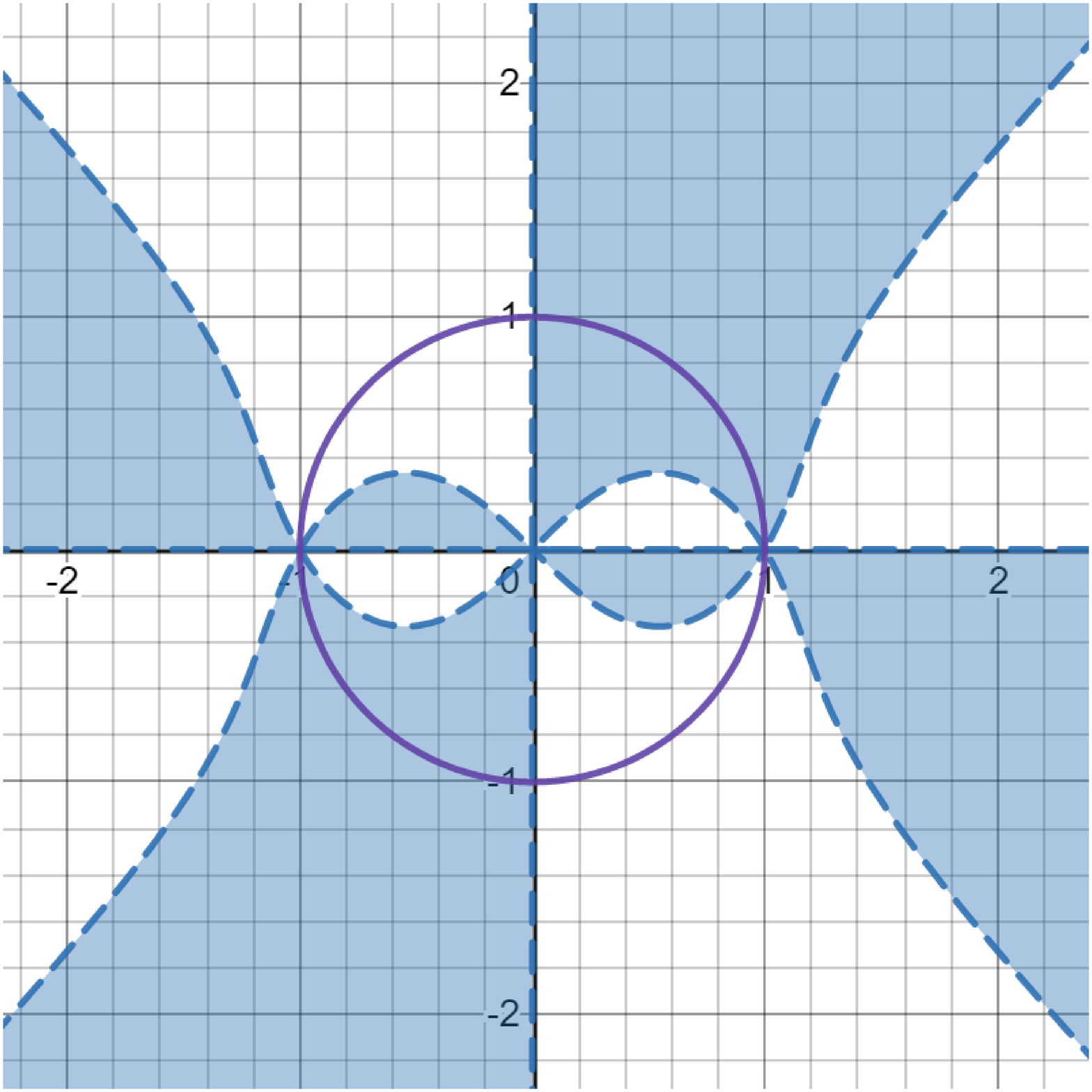}}}~~
{\rotatebox{0}{\includegraphics[width=3.0cm,height=2.75cm,angle=0]{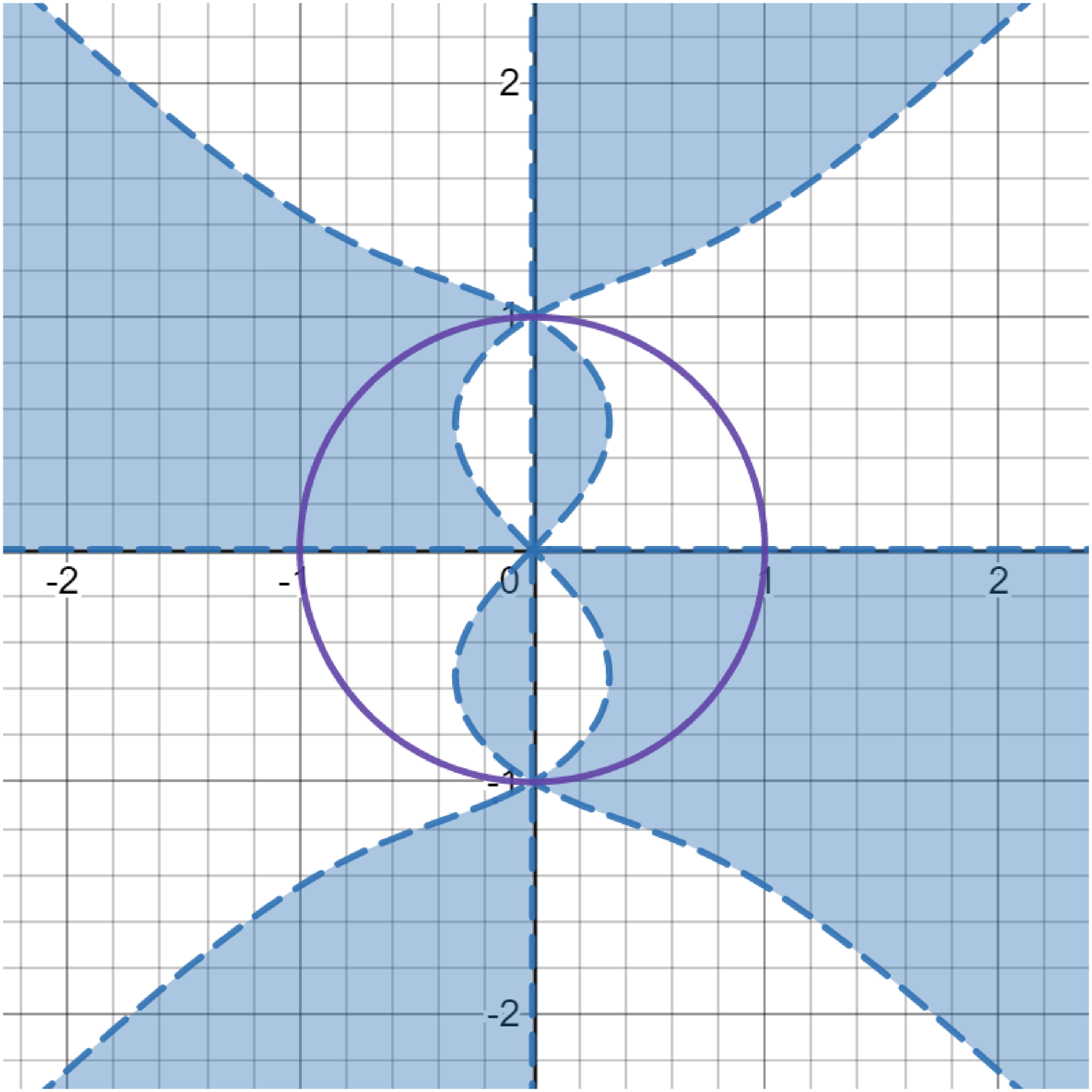}}}~~
{\rotatebox{0}{\includegraphics[width=3.0cm,height=2.75cm,angle=0]{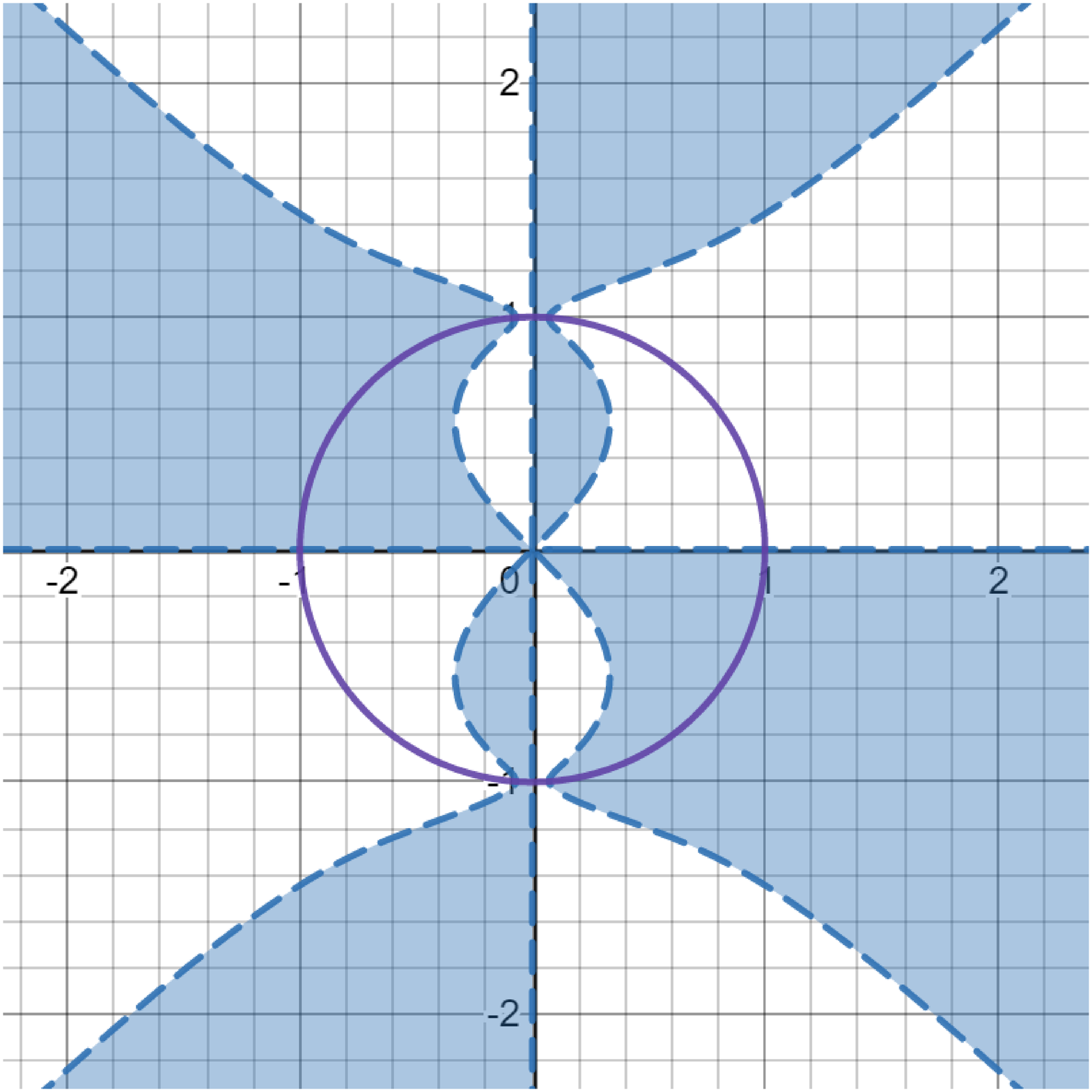}}}}

$\qquad\qquad\quad\quad(\textbf{a})\qquad \qquad\qquad\qquad(\textbf{b})
\qquad\qquad\quad\qquad\quad(\textbf{c})$\\

\centerline{
{\rotatebox{0}{\includegraphics[width=3.0cm,height=2.75cm,angle=0]{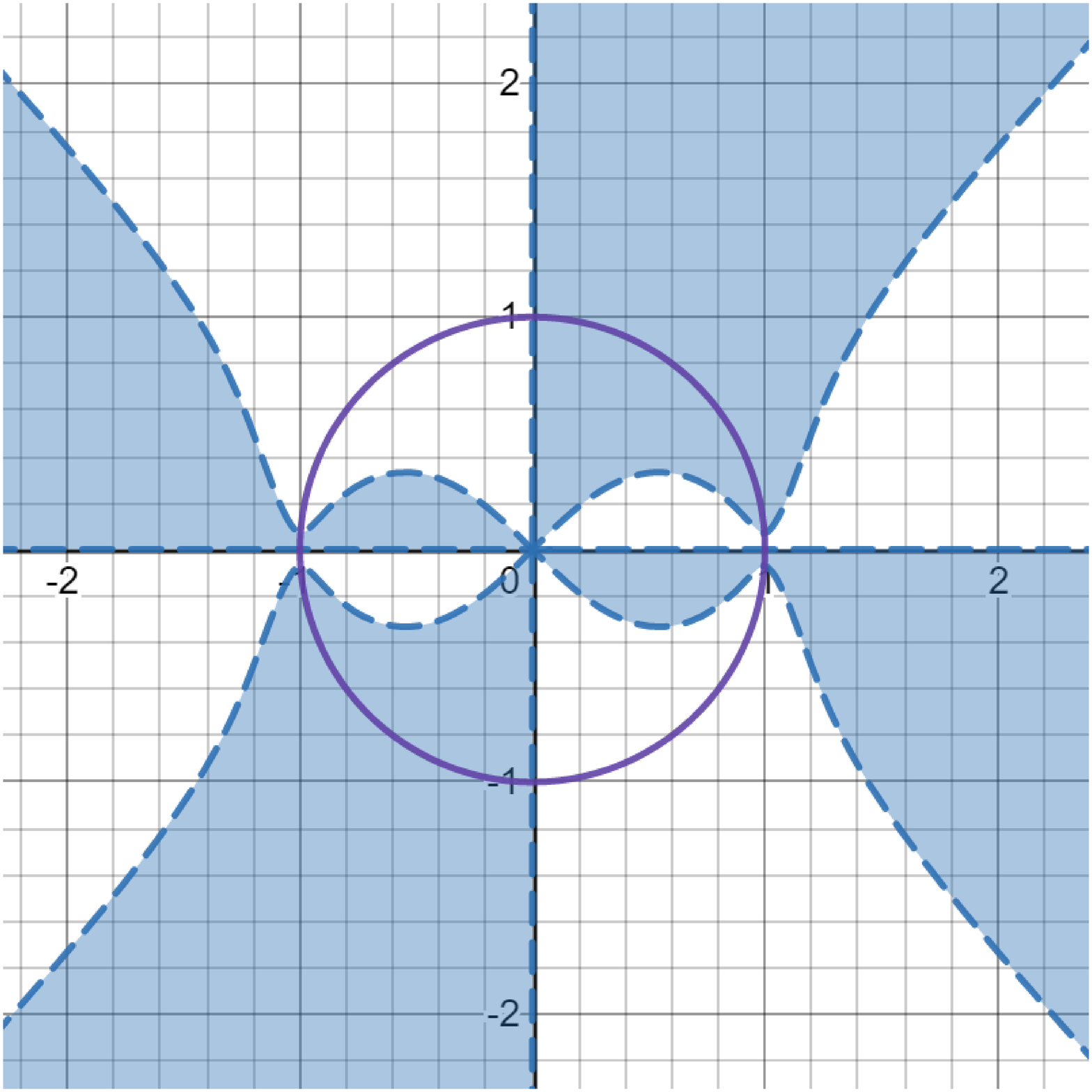}}}~~
{\rotatebox{0}{\includegraphics[width=3.0cm,height=2.75cm,angle=0]{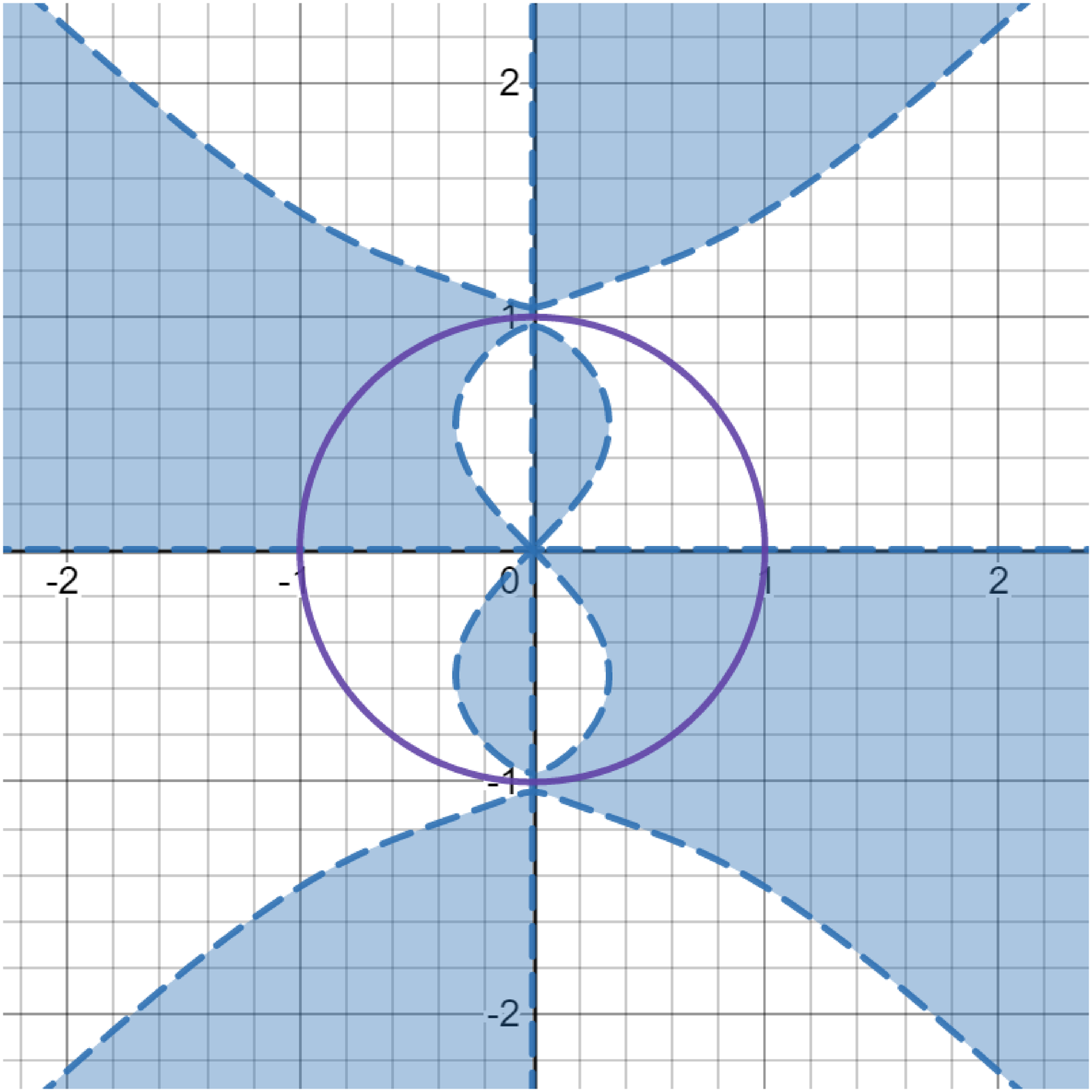}}}~~
{\rotatebox{0}{\includegraphics[width=3.0cm,height=2.75cm,angle=0]{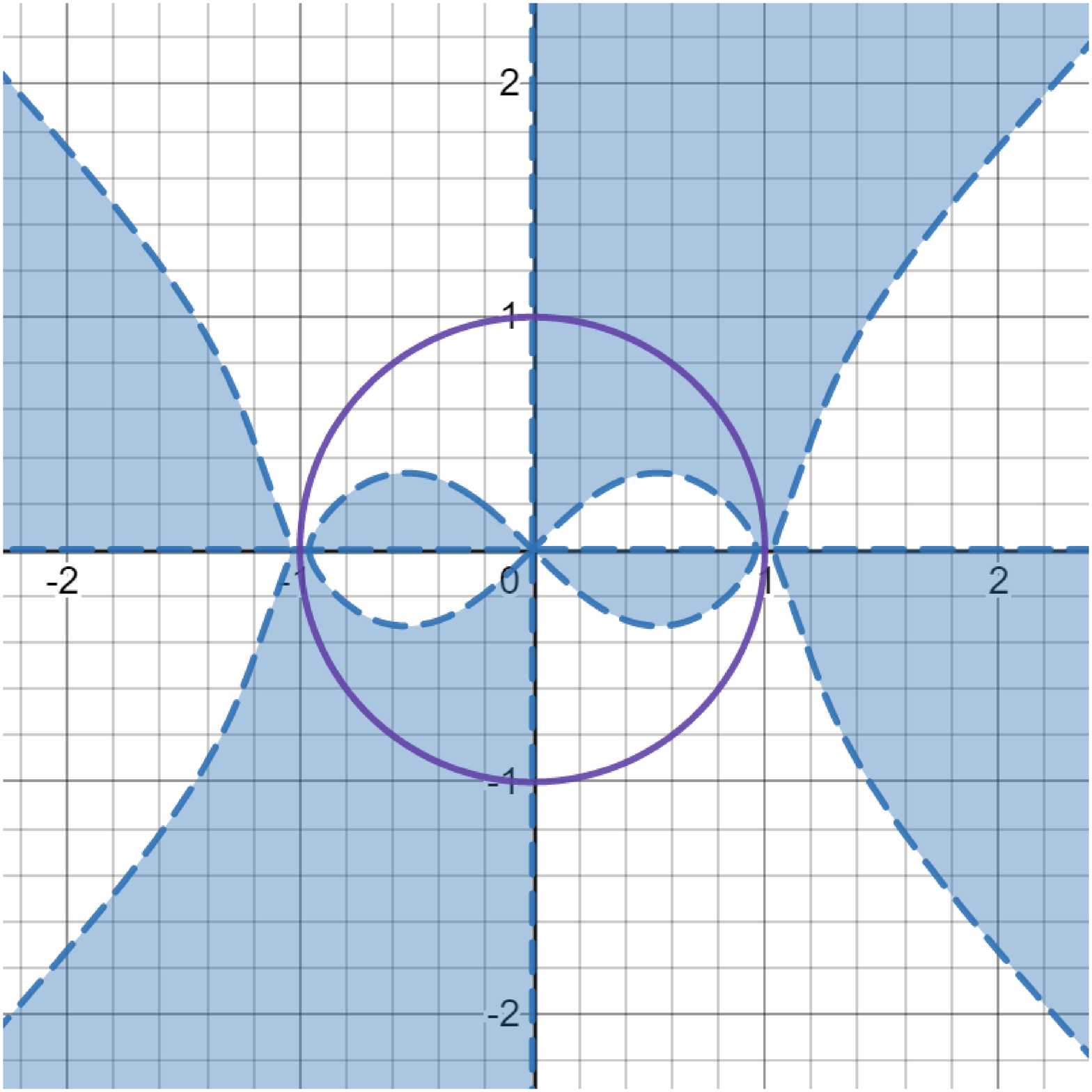}}}}

$\qquad\qquad\quad\quad(\textbf{d})\qquad \qquad\qquad\qquad(\textbf{e})
\qquad\qquad\quad\qquad\quad(\textbf{f})$\\

\centerline{
{\rotatebox{0}{\includegraphics[width=3.0cm,height=2.75cm,angle=0]{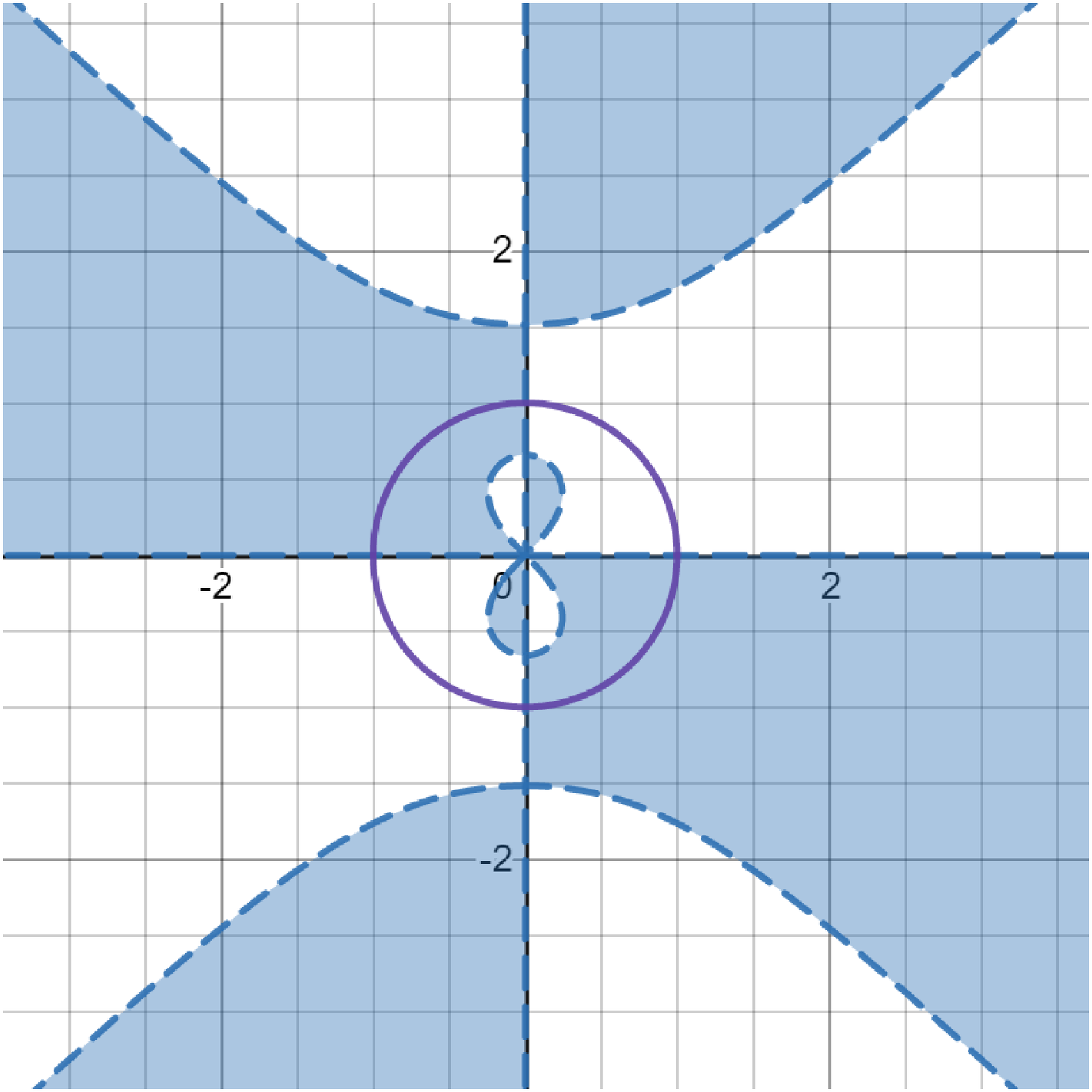}}}~~
{\rotatebox{0}{\includegraphics[width=3.0cm,height=2.75cm,angle=0]{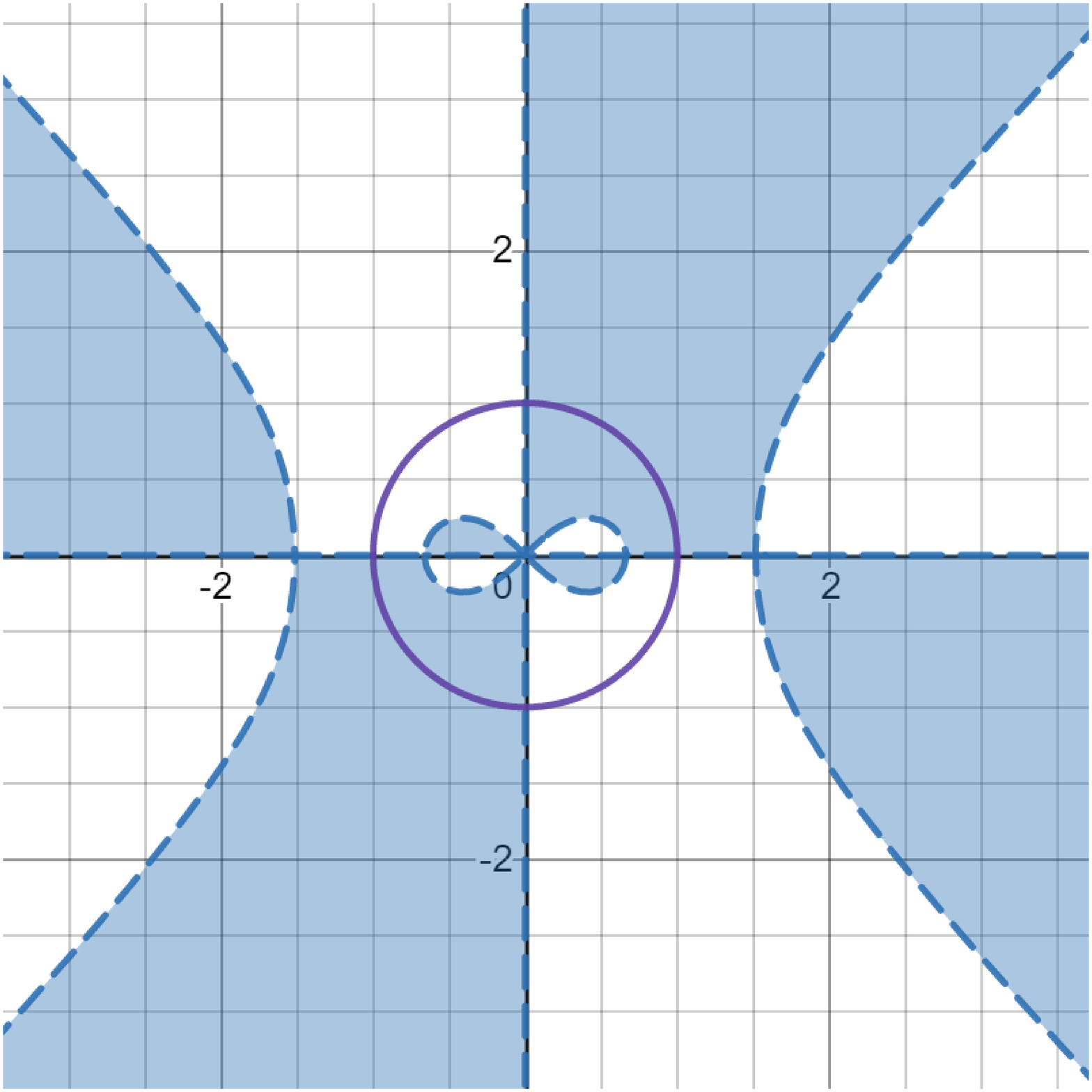}}}~~
{\rotatebox{0}{\includegraphics[width=3.0cm,height=2.75cm,angle=0]{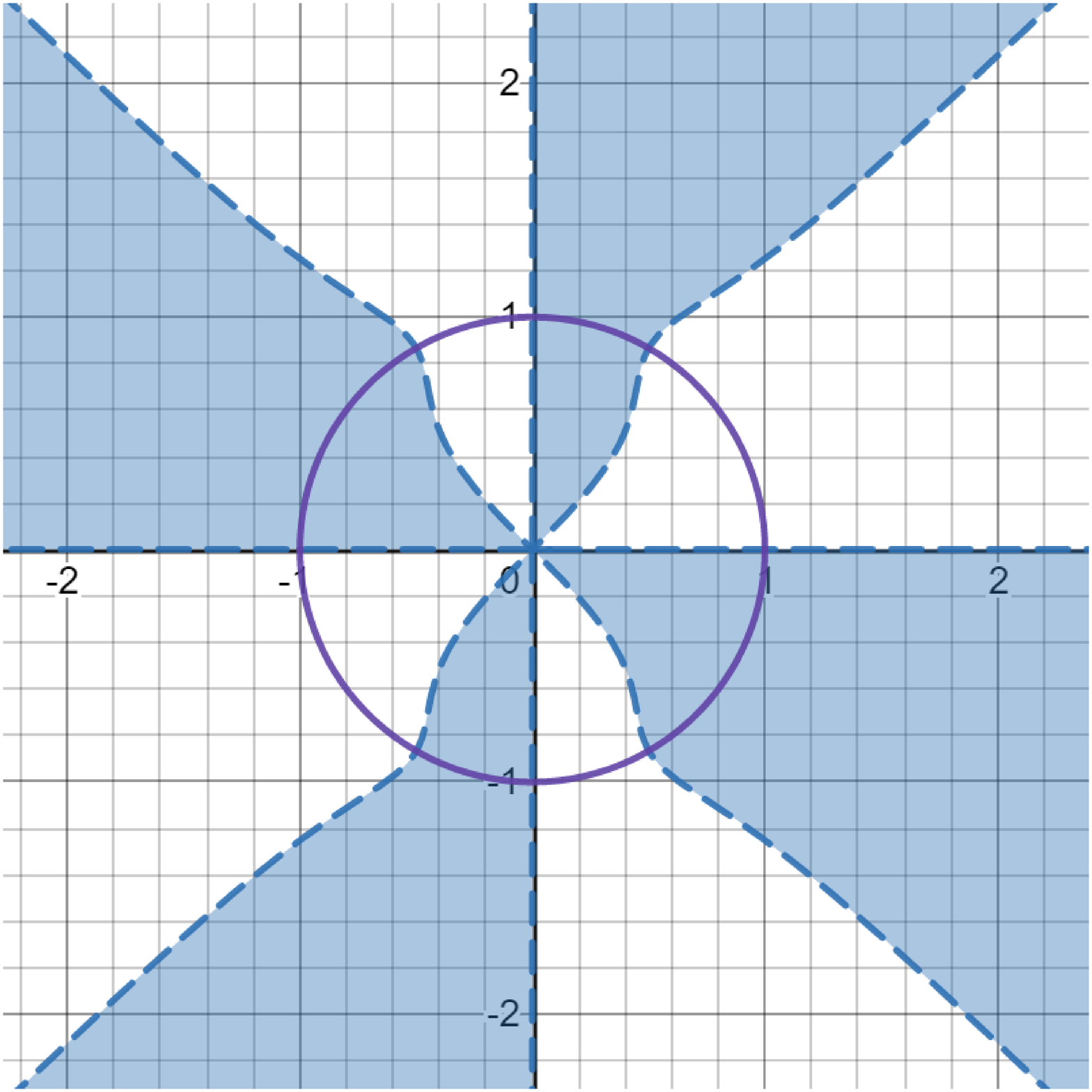}}}}

$\qquad\qquad\quad\quad(\textbf{g})\qquad \qquad\qquad\qquad(\textbf{h})
\qquad\qquad\quad\qquad\quad(\textbf{i})$\\

\noindent {\small \textbf{Figure 3.} The parametes $\xi=-5,7,6.99,5.99,7.01,5.01,8,4,6.5$
respectively to show all type of $\text{Im}\theta$.  In the blue region,
$\text{Im}\theta>0$, $\text{Im}\theta<0$ in the white region and the dashed line indicates that $\text{Im}\theta=0$.}\\

Note that the modulus of the oscillating term $e^{2it\theta}$ is $e^{Re(2it\theta_n)}$, naturally, there are three situations for the pole $z\in\mathcal {Z}$: those for $Re(2it\theta_n)>0$, corresponding to a connection coefficient $C_ne^{2it\theta_{n}}$ is exponentially large for $t\rightarrow\infty$; those for $Re(2it\theta_n)<0$, corresponding to a connection coefficient $C_ne^{2it\theta_{n}}$ is exponentially decaying for $t\rightarrow\infty$; those for $Re(2it\theta_n)=0$, the connection coefficient is bounded in time. In the next analysis process, we will deal with these three cases differently. For the exponential growth and decay discrete spectrum, we interpolate them to make the exponential growth discrete spectrum trade into a jump on an arbitrary small disk, while the discrete spectrum on the dotted line $Re(2it\theta_n)=0$ mainly produces soliton solutions.

In our work, we limit the  function $\xi$ to the interval $(5,7)$, which corresponds to figure $(i)$. For convenience,  the following notations are introduced
\begin{align}
&\mathcal{N}\triangleq\left\lbrace 1,\ldots,4N_1+2N_2\right\rbrace, \ \ \nabla=\left\lbrace n \in  \mathcal{N}  |\text{Im}\theta_n\leq 0\right\rbrace,\nonumber\\
&\Delta=\left\lbrace n \in  \mathcal{N} |\text{Im}\theta_n> 0\right\rbrace,\quad\Lambda=\left\lbrace n \in  \mathcal{N}  |\text{Im}\theta_n<\varepsilon_0\right\rbrace,\label{devide}
\end{align}
where $\varrho_0=\min_{n\in\mathcal {N}\backslash\Lambda}\{|Im\theta|\}>\varepsilon_0$ is a sufficiently small constant. In order to distinguish the attenuation region of the oscillation term, the interval is further divided as follows
\begin{align}
	&\nabla_1=\left\lbrace j \in \left\lbrace 1,\ldots,N_1\right\rbrace  |\text{Im}\theta(z_j)< 0\right\rbrace,
	~~\nabla_2=\left\lbrace i \in \left\lbrace 1,\ldots,N_2\right\rbrace  |\text{Im}\theta(w_i)< 0\right\rbrace,\nonumber\\
	&\Delta_1=\left\lbrace j \in \left\lbrace 1,\ldots,N_1\right\rbrace  |\text{Im}\theta(z_j)> 0\right\rbrace,
	~~\Delta_2=\left\lbrace i \in \left\lbrace 1,\ldots,N_2\right\rbrace  |\text{Im}\theta(w_i)> 0\right\rbrace,\nonumber\\
	&\Lambda_1=\left\lbrace j_0 \in \left\lbrace 1,\ldots,N_1\right\rbrace  |\text{Im}\theta(z_{j_0})\leq \varepsilon_0\right\rbrace,~~\Lambda_2=\left\lbrace i_0 \in \left\lbrace 1,\ldots,N_2\right\rbrace  |\text{Im}\theta(w_{i_0})\leq \varepsilon_0\right\rbrace.\nonumber
\end{align}

The next step is to use two known matrix trigonometric factorizations for the jump matrix $V(z)$ in \eqref{jumpv}
\begin{align}\label{jumpv1}
&V(z)=\left(\begin{array}{cc}
1-\widetilde{\rho}(z)\rho(z) & -e^{2it\theta}\widetilde{\rho}(z)\\
e^{-2it\theta}\rho(z) & 1
\end{array}\right)=U_L(z)U_R(z)=W_L(z)T_{0}(z)W_R(z),\\
&U_L(z)=\left(\begin{array}{cc}
	1 & -\widetilde{\rho}(z) e^{2it\theta}\\
	0 & 1
\end{array}\right),\quad U_R(z)=\left(\begin{array}{cc}
1 & 0 \\
\rho(z) e^{-2it\theta} & 1
\end{array}\right),\\
&W_L(z)=\left(\begin{array}{cc}
	1 & \\
	\frac{\rho(z) e^{-2it\theta}}{1-\rho(z)\widetilde{\rho}(z)} & 1
\end{array}\right),~ T_{0}(z)=(1-\rho(z)\widetilde{\rho}(z))^{\sigma_3},~ W_R(z)=\left(\begin{array}{cc}
1 & -\frac{\widetilde{\rho}(z) e^{2it\theta}}{1-\rho(z)\widetilde{\rho}(z)} \\
0 & 1
\end{array}\right).
\end{align}

Note that for $\zeta_n$ with $n\in\Delta$, the residue conditions at these poles are unbounded as $t\rightarrow\infty$ and the purpose of nonlinear steepest descent method is to exchange these oscillatory terms appearing in residues into new jump matrices so that they are asymptotically zero when $t$ approaches infinity. Now define the function
\begin{align}\label{T}
T(z)&=T(z,\xi)=\prod_{n\in \Delta}\dfrac{z-\zeta_n}{\bar{\zeta}_n^{-1}z-1}\delta (z)
=\prod_{j\in \Delta_1}\dfrac{z^2-z_j^2}{\bar{z}_j^{-2}z^2-1}\dfrac{z^2-\bar{z}_j^{-2}}{z_j^{2}z^2-1}\prod_{\ell\in \Delta_2}\dfrac{z^2-w_\ell^2}{w_\ell^{2}z^2-1}\delta (z),
\end{align}
where
\begin{align*}
\delta (z)&=\exp\left(-\frac{1}{2\pi i}\int _{i\mathbb{R}}\left( \dfrac{1}{s-z}-\frac{1}{2s}\right) \log (1-\rho(s)\widetilde{\rho}(s))ds\right).
\end{align*}
\begin{lem}\label{lem3}
The function $T(z)$ in \eqref{T} is meromorphic in $\mathbb{C}\backslash\mathbb{R}$ with simple poles at the $\eta_{n}$ and simple zeros at the $\overline{\eta}_{n}$ for $n\in\Delta$, and admits the following conditions
\begin{enumerate}[(1)]
\item Jump condition: for $z\in\mathbb{R}$
\begin{align}\label{TJ}
T_{+}(z)=(1-\rho(z)\widetilde{\rho}(z))T_{-}(z).
\end{align}
\item Symmetries:
\begin{align}\label{ST}
T(z)=\overline{T^{-1}(\bar{z})}=T^{-1}(-z^{-1}).
\end{align}
\item The limit of the function $T(z)$
\begin{align}
T(\infty):=\lim_{z\rightarrow\infty}T(z)=\prod_{j\in \Delta_1} \overline{z}_j^2z_j^{-2}\prod_{\ell\in \Delta_2}\overline{w}_\ell^{2}\exp\left(\frac{1}{4\pi i}\int_{i\mathbb{R}}s^{-1}  \log (1-\rho(s)\widetilde{\rho}(s))ds\right).
	\end{align}
\item As $z\rightarrow\infty$, the asymptotic expansion for $T(z)$
\begin{align}
	T(z)=T(\infty)\left( 1+z^{-1}\frac{1}{2\pi i}\int _{i\mathbb{R}}\log (1-\rho(s)\widetilde{\rho}(s))ds+ \mathcal{O}(z^{-2})\right) \label{expT}.
	\end{align}
\item $T(z) $ is  continuous at $z=0$ with
\begin{align}
	\lim_{z\to 0}T(z)=T(0)=T(\infty)^{-1} \label{T0}.
	\end{align}
\item There exists a constant $A(q_0)$ such that
\begin{align}
\left|\frac{s_{11}(z)}{T(z)}\right|\leq A(q_0),
\end{align}
which is meromorphic function in $D^+\cap \left\lbrace z\in\mathbb{C}|\text{Re} z>0 \right\rbrace $.
\end{enumerate}
\end{lem}
\begin{proof}
Using the definition of $T(z)$ in \eqref{T}, it is obvious that $T(z)$ has simple poles at $\overline{\zeta}_{n}$ and simple zeros at $\zeta_n$ for $n\in\Delta$. The jump condition \eqref{TJ} follows from  the Plemej formula. To show the symmetries
\eqref{ST}, we denote $T(z)=T_1(z)T_2(z)T_3(z)$. For $T_{1}(z)$
\begin{align*}
T_{1}(\overline{z})=\prod_{j\in \Delta_1}\dfrac{\overline{z}^2-z_j^2}{\overline{z}_j^{-2}\overline{z}^2-1}\dfrac{\overline{z}^2-
\overline{z}_j^{-2}}{z_j^{2}z^2-1}\Longrightarrow  T_{1}^{-1}(\overline{z})=\prod_{j\in \Delta_1}\frac{\overline{z}^2-z_j^2}{\overline{z}_j^{-2}\overline{z}^2-1}\frac{\overline{z}^2-z_j^{-2}}
{\overline{z}_j^{2}\overline{z}^2-1},
\end{align*}
which implies that $T_1(z)=\overline{T_1^{-1}(\bar{z})}$. Similarly, $T_2(z)=\overline{T_2^{-1}(\bar{z})}$ and $T_3(z)=\overline{T_3^{-1}(\bar{z})}$ can be proved. Finally the first equality in the symmetry \eqref{ST} can be obtained. The second equality follows from the symmetry \eqref{RS}. By using Laurent expansion, the claim 3 is derived immediately. The claims 4 and 5 can be obtained by simple calculation.  Finally, taking the ratio $\frac{s_{11}(z)}{T(z)}$ into account one has
\begin{align*}
\frac{s_{11}(z)}{T(z)}=&\prod_{j\in \Delta_1} \overline{z}_j^{-2}z_j^{2}\prod_{\ell\in \Delta_2}\overline{w}_\ell^{-2}e^{\left(-\frac{1}{4\pi i}\int_{i\mathbb{R}}s^{-1}  \log (1-\rho(s)\widetilde{\rho}(s))ds\right)}\times\\
&\prod_{j\in \nabla_1}\dfrac{z^2-z_j^2}{\bar{z}_j^{-2}z^2-1}\dfrac{z^2-\bar{z}_j^{-2}}{z_j^{2}z^2-1}\prod_{\ell\in \nabla_2}\dfrac{z^2-w_\ell^2}
{w_\ell^{2}z^2-1}\exp\left\lbrace -\frac{1}{2\pi i}\int _{\mathbb{R}}\frac{\log (1-\rho(s)\tilde{\rho}(s))}{s-z}\right\rbrace.
\end{align*}
Note that the module before the last factors are obviously $\leq1$ for $z\in D^{+}$ while the real part of the last factor can be estimated as follows
\begin{align*}
&\left|-\frac{Imz}{2\pi}\int _{\mathbb{R}}\frac{\log (1+|\rho(s)|^2)}{|s-z|^2}ds\right|=\left|-\frac{Imz}{2\pi}\int _{\mathbb{R}}\frac{\log (1+|\rho(s)|^2)}{(s-Rez)^{2}+Imz^2}ds\right|\\ &\leq\frac{1}{2\pi}\left|\left|\log (1+|\rho(s)|^2)\right|\right|_{L^{\infty}(\mathbb{R})}\left|\left|\frac{Imz}{(s-Rez)^{2}+Imz^2}\right|
\right|_{L^{1}(\mathbb{R})}.
\end{align*}
The first term of the above equation can be shown to be bounded using Lemma 4.5 in \cite{SandRNLS}, and the second term can be verified to be bounded by a simple calculation.
\end{proof}

Now performing the above analysis, i.e., by interpolating the poles we trade the poles to any small disk enclosing these poles and the new jumps are bounded as time tends to infinity. Define
\begin{align}\label{M1}
M^{(1)}(z)=\left\{ \begin{array}{ll}
		T(\infty)^{-\sigma_3}M(z)\left(\begin{array}{cc}
			1 & 0\\
			-C_n(z-\eta_n)^{-1}e^{-2it\theta_n} & 1
		\end{array}\right)T^{\sigma_3}(z),   &\text{as } |z-\eta_n|<\epsilon_0,~n\in\nabla\setminus\Lambda,\\[12pt]
		T(\infty)^{-\sigma_3}M(z)\left(\begin{array}{cc}
			1 & -C_n^{-1}(z-\eta_n)e^{2it\theta_n}\\
			0 & 1
		\end{array}\right)T^{\sigma_3}(z),   &\text{as } |z-\eta_n|<\epsilon_0,~n\in\Delta,\\
		T(\infty)^{-\sigma_3}M(z)\left(\begin{array}{cc}
		1 & \overline{C}_n(z-\overline{\eta}_n)^{-1}e^{2it\overline{\theta}_n}\\
		0 & 1
		\end{array}\right)T^{\sigma_3}(z),   &\text{as } 	|z-\overline{\eta}_n|<\epsilon_0,~n\in\nabla\setminus\Lambda,\\
		T(\infty)^{-\sigma_3}M(z)\left(\begin{array}{cc}
		1 & 0	\\
		\overline{C}_n^{-1}(z-\overline{\eta}_n)e^{-2it\overline{\theta}_n} & 1
		\end{array}\right)T^{\sigma_3}(z),   &\text{as } 	|z-\overline{\eta}_n|<\epsilon_0,~n\in\Delta,\\
	T(\infty)^{-\sigma_3}M(z)T^{\sigma_3}(z) &\text{as } 	z \text{ in elsewhere},
	\end{array}\right.
\end{align}
where  $\epsilon_0$  is a positive constant with
\begin{align}
\epsilon_0=\frac{1}{2}\min\left\lbrace\min_{ j\neq i\in \mathcal{N}}|\eta_i-\eta_j|, \min_{j\in \mathcal{N}}\left\lbrace |\text{Im}\eta_j|, |\text{Re}\eta_j|\right\rbrace ,\min_{j\in \mathcal{N}\setminus\Lambda,\text{Im}\theta(z)=0}|\eta_j-z| \right\rbrace.
\end{align}
Consider the contour shown in Figure 4:
\begin{align}
\Sigma^{(1)}=\mathbb{R}\cup i\mathbb{R}\cup\left[\cup_{n\in\mathcal{N}\setminus\Lambda}\left( |z-\eta_n|=\epsilon_0\cup |z-\overline{\eta}_n|=\epsilon_0\right)  \right].
\end{align}
\\
\centerline{{\rotatebox{0}{\includegraphics[width=8.0cm,height=7.75cm,angle=0]{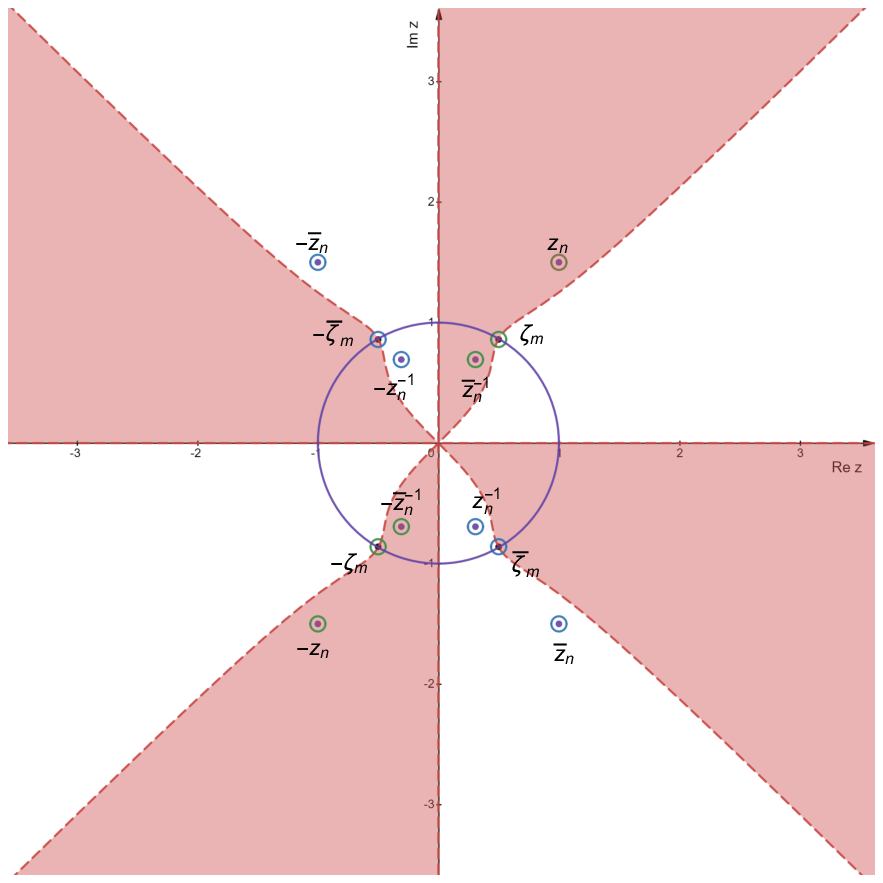}}}}
\centerline{\noindent {\small \textbf{Figure 4.} The contour $\Sigma^{(1)}$, the dotted line indicates that $Im\theta=0$.}}\\

\begin{RHP}\label{RHP2}
Find a $2\times2$ matrix-valued function $M^{(1)}(z)$ such that
\begin{enumerate}[(1)]
\item The function $M^{(1)}(z)$ is meromorphic in $\mathbb{C}\backslash\Sigma^{(1)}$.
\item
$M^{(1)}(z)=e^{i\nu_{-}(x,q)\sigma_{3}}+\mathcal {O}(z^{-1}),\quad\qquad z\rightarrow\infty,$\\
$M^{(1)}(z)=\frac{1}{z}e^{i\nu_{-}(x,q)\sigma_{3}}\sigma_{3}Q_{-}+\mathcal {O}(1),\quad z\rightarrow0.$
\item The symmetries:
$M^{(1)}(z)=\sigma_2\overline{M^{(1)}(\bar{z})}\sigma_2$=$\sigma_1\overline{M^{(1)}(-\bar{z})}
\sigma_1=\frac{i}{z}M^{(1)}(-1/z)\sigma_3Q_-$.
\item  For $z\in\Sigma^{(1)}$, there exists the non-tangential boundary values $M_{\pm}^{(1)}(z)$ such that
\begin{align}
	M^{(1)}_+(z)=M^{(1)}_-(z)V^{(1)}(z),
\end{align}
with
\begin{align}
	V^{(1)}(z)=\left\{\begin{array}{ll}\left(\begin{array}{cc}
		1 & -e^{2it\theta}\widetilde{\rho}(z)T^{-2}(z) \\
		0 & 1
	\end{array}\right)
		\left(\begin{array}{cc}
			1 & 0\\
			e^{-2it\theta}\rho(z)T^2(z) & 1
		\end{array}\right),   ~~\text{as } z\in 	\mathbb{R},\\[12pt]
		\left(\begin{array}{cc}
		1 & 0\\
		\frac{e^{-2it\theta} \rho(z)T_+^{2}(z)}{1-\widetilde{\rho}(z)\rho(z)} & 1
	\end{array}\right)\left(\begin{array}{cc}
	1 & -\frac{e^{2it\theta}\widetilde{\rho}(z)T_-^{-2}(z)}{1-\widetilde{\rho}(z)\rho(z)}\\
	0 & 1
\end{array}\right),  ~~~ ~~\text{as } z\in i\mathbb{R},\\[12pt]
		\left(\begin{array}{cc}
		1 & 0\\
		-C_n(z-\eta_n)^{-1}T^2(z)e^{-2it\theta_n} & 1
		\end{array}\right),  ~~~ \text{as } 	|z- \eta_{n}|=\epsilon_0,~~n\in\nabla\setminus\Lambda,\\[12pt]
		\left(\begin{array}{cc}
			1 & -C_n^{-1}(z-\eta_n)T^{-2}(z)e^{2it\theta_n}\\
			0 & 1
		\end{array}\right),  ~~~~ \text{as } |z- \eta_{n}|=\epsilon_0,~~n\in\Delta,\\
		\left(\begin{array}{cc}
			1 & \overline{C}_n(z-\overline{\eta}_n)^{-1}T^{-2}(z)e^{2it\overline{\theta}_n}\\
			0 & 1
		\end{array}\right),  ~~~~ \text{as } 	|z-\overline{\eta}_{n}|=\epsilon_0,~~n\in\nabla\setminus\Lambda,\\
		\left(\begin{array}{cc}
			1 & 0	\\
			\overline{C}_n^{-1}(z-\overline{\eta}_n)e^{-2it\bar{\theta}_n}T^2(z) & 1
		\end{array}\right),  ~~~~~~ \text{as } 	|z-\overline{\eta}_{n}|=\epsilon_0,~~n\in\Delta.\\
	\end{array}\right.
\end{align}
\item Residue conditions: for $n\in\Lambda$ the function $M^{(1)}(z)$ defined in \eqref{M1} has simple poles at $\eta_{n}$ and $\overline{\eta}_{n}$
\begin{align}
	&\res_{z=\eta_n}M^{(1)}(z)=\lim_{z\to \eta_n}M^{(1)}(z)\left(\begin{array}{cc}
		0 & 0\\
		C_ne^{-2it\theta_n}T^2(\eta_n) & 0
	\end{array}\right),\\
	&\res_{z=\overline{\eta}_n}M^{(1)}(z)=\lim_{z\to \overline{\eta}_n}M^{(1)}(z)\left(\begin{array}{cc}
		0 & -\overline{C}_nT^{-2}(\overline{\eta}_n)e^{2it\overline{\theta}_n}\\
		0 & 0
	\end{array}\right).
\end{align}
\end{enumerate}
\end{RHP}
\begin{proof}
The analytic property, asymptotic property and symmetry of $M^{(1)}(z)$ inherit the related properties of $M(z)$ satisfied by the RH problem \ref{RHP1}. The jump conditions on real axis and imaginary axis follows from the related references \cite{DZ-1993,DZ-1994}. Take the jump condition on the disk $|z-\eta_n|=\epsilon_0$ as example, the other cases can be proved by the same method. Take counter clockwise as positive direction, for $z\in\{|z-\eta_n|=\epsilon_0\}$ with $n\in\Delta\backslash\Lambda$
\begin{align*}
T(\infty)^{-\sigma_3}M(z)\left(\begin{array}{cc}
			1 & 0\\
			-C_n(z-\eta_n)^{-1}e^{-2it\theta_n} & 1
		\end{array}\right)T(z)^{\sigma_3}=T(\infty)^{-\sigma_3}M(z)T(z)^{\sigma_3}V^{(1)}(z),
\end{align*}
which implies that
\begin{align*}
V^{(1)}(z)=\left(\begin{array}{cc}
			1 & 0\\
			-C_n(z-\eta_n)^{-1}T(z)^2e^{-2it\theta_n} & 1
		\end{array}\right).
\end{align*}
Finally, we consider the residue conditions at the poles $\eta_n$ and $\overline{\eta}_n$. There are the simple poles at $\overline{\eta}_n$ and simple zeros at $\eta_n$ from the definition of function $T(z)$ defined in \eqref{T}. Then
\begin{align*}
\res_{z=\eta_n}M^{(1)}(z)&=\res_{z=\eta_n}T(\infty)^{-\sigma_3}M(z)T(z)^{\sigma_3}\\
&=\lim_{z\rightarrow\eta_n}T(\infty)^{-\sigma_3}M(z)T(z)^{\sigma_3}T(z)^{-\sigma_3}
\left(\begin{array}{cc}
			0 & 0\\
			C_ne^{-2it\theta_n} & 0
		\end{array}\right)T(z)^{\sigma_3}\\
&=\lim_{z\rightarrow\eta_n}M^{(1)}(z)\left(\begin{array}{cc}
			0 & 0\\
			C_ne^{-2it\theta_n}T(\eta_n)^2 & 0
		\end{array}\right)T(z)^{\sigma_3}.
\end{align*}
On the other hand,
\begin{align*}
\res_{z=\overline{\eta}_n}M^{(1)}(z)&=\res_{z=\overline{\eta}_n}T(\infty)^{-\sigma_3}M(z)T(z)^{\sigma_3}\\
&=\lim_{z\rightarrow\overline{\eta}_n}T(\infty)^{-\sigma_3}M(z)T(z)^{\sigma_3}T(z)^{-\sigma_3}
\left(\begin{array}{cc}
			0 & -\overline{C}_ne^{2it\overline{\theta}_n}\\
			0 & 0
		\end{array}\right)T(z)^{\sigma_3}\\
&=\lim_{z\rightarrow\overline{\eta}_n}M^{(1)}(z)\left(\begin{array}{cc}
			0 & -\overline{C}_ne^{2it\overline{\theta}_n}T(\overline{\eta}_n)^{-2}\\
			0 & 0
		\end{array}\right)T(z)^{\sigma_3}.
\end{align*}
\end{proof}
\begin{rem}
Compared with $M(z)$ satisfied the RH problem \ref{RHP1}, the matrix-valued function $M^{(1)}(z)$ in \eqref{M1}  defined by function $T(z)$ in \eqref{T} has the following advantages:
\begin{enumerate}[(1)]
\item The factor Blaschke product appearing in the function $T$ in \eqref{T}
exchanges the column in which the poles $\eta_j$, $j\in\Delta$ appear and establishes the desired result that the oscillation term is bounded when time $t$ approaches infinity.
\item The triangular factors appearing in the matrix valued function $M^{(1)}(z)$ defined by \eqref{M1} trade the poles to the jumps on the boundaries of disk  $|z-\eta_j|=\epsilon_0$.
\item The factor $(\overline{\eta}^{-1}_nz-1)$ of the Blaschke product in function $T$ is different from that $(z-\overline{z}_n)$ in \cite{SandRNLS}, and the purpose of adding item $1/2s$ is to ensure that $T$ satisfies the symmetry \eqref{ST} in the lemma \ref{lem3}.
\end{enumerate}
\end{rem}
\subsection{Opening lenses}\quad
Now we in this section make continuous extension to the jump matrix $V^{(1)}(z)$ and remove the jump from the interval $\Sigma$ in such a way that the new problem makes full use of the growth and decay properties of the oscillation term $e^{2it\theta(z)}$ for $z\notin\Sigma$. In addition, in order to remove the poles in the problem, we open the lens in this way that the lens is far away from the previously introduced disk. To achieve this goal, we first introduce the following concepts.
Define $\Omega=\cup_{k=1}^{8}\Omega_{k}$ with
\begin{subequations}
\begin{align}
&\Omega_{1}=\{z:arg\in(0,\psi)\},\qquad \Omega_{2}=\{z:arg\in(\pi/2-\psi,\pi/2)\},\\
&\Omega_{3}=\{z:arg\in(\pi/2,\pi/2+\psi)\},\quad \Omega_{4}=\{z:arg\in(\pi-\psi,\pi)\},\\
&\Omega_{5}=\{z:arg\in(\pi,\pi+\psi)\},\quad \Omega_{6}=\{z:arg\in(3\pi/2-\psi,3\pi/2)\},\\
&\Omega_{7}=\{z:arg\in(3\pi/2,3\pi/2+\psi)\},\quad \Omega_{8}=\{z:arg\in(2\pi-\psi,2\pi)\},
\end{align}
\end{subequations}
where $\psi>0$ is a sufficiently small angle and satisfies
\begin{enumerate}[(1)]
\item $\frac{2|\xi-6|}{|\xi-6|+1}<\cos\psi<1$.
\item The regions $\Omega_{k}$  don't intersect with any disk $|z-\eta_n|=\epsilon_0$ and $|z-\overline{\eta}_n|=\epsilon_0$.
\end{enumerate}
Finally, the following contours are defined
\begin{subequations}
\begin{align}
&\Sigma_k=e^{(k-1)i\pi/4+\psi}\mathbb{R}^+,\hspace{0.5cm}k=1,3,5,7,\\
&\Sigma_k=e^{ki\pi/4-\psi}\mathbb{R}^+,\hspace{1.0cm}k=2,4,6,8,\\
&\widetilde{\Sigma}=\Sigma_1\cup\Sigma_2\ldots\cup\Sigma_{8},
\end{align}
\end{subequations}
which are shown in Fig. 5.\\

\centerline{
		\begin{tikzpicture}[node distance=2cm]
		\draw[gray, fill=gray!40] (0,0)--(0.5,3)--(-0.5,3)--(0,0)--(0.5,-3)--(-0.5,-3)--(0,0);
		\draw[gray, fill=gray!40] (0,0)--(3,-0.5)--(3,0.5)--(0,0)--(-3,-0.5)--(-3,0.5)--(0,0);
		\draw(0,0)--(3,0.5)node[above]{$\Sigma_1$};
		\draw(0,0)--(0.5,3)node[right]{$\Sigma_2$};
		\draw(0,0)--(-0.5,3)node[left]{$\Sigma_3$};
		\draw(0,0)--(-3,0.5)node[left]{$\Sigma_4$};
		\draw(0,0)--(-3,-0.5)node[left]{$\Sigma_5$};
		\draw(0,0)--(-0.5,-3)node[left]{$\Sigma_6$};
		\draw(0,0)--(0.5,-3)node[right]{$\Sigma_7$};
		\draw(0,0)--(3,-0.5)node[right]{$\Sigma_8$};
		\draw[->](-3.5,0)--(3.5,0)node[right]{ Re$z$};
		\draw[->](0,-3.5)--(0,3.5)node[above]{ Im$z$};
		\draw[-latex](0,0)--(-1.5,-0.25);
		\draw[-latex](0,0)--(-1.5,0.25);
		\draw[-latex](0,0)--(1.5,0.25);
		\draw[-latex](0,0)--(1.5,-0.25);
		\draw[-latex](0,0)--(0.25,-1.5);
		\draw[-latex](0,0)--(0.25,1.5);
		\draw[-latex](0,0)--(-0.25,1.5);
		\draw[-latex](0,0)--(-0.25,-1.5);
		\coordinate (I) at (0.2,0);
		\coordinate (C) at (-0.2,2.2);
		\fill (C) circle (0pt) node[above] {\small $\Omega_3$};
		\coordinate (E) at (0.2,2.2);
		\fill (E) circle (0pt) node[above] {\small $\Omega_2$};
		\coordinate (D) at (2.2,0.2);
		\fill (D) circle (0pt) node[right] {\small$\Omega_1$};
		\coordinate (F) at (-0.2,-2.2);
		\fill (F) circle (0pt) node[below] {\small$\Omega_6$};
		\coordinate (J) at (-2.2,-0.2);
		\fill (J) circle (0pt) node[left] {\small$\Omega_5$};
		\coordinate (k) at (-2.2,0.2);
		\fill (k) circle (0pt) node[left] {\small$\Omega_4$};
		\coordinate (J) at (0.2,-2.2);
		\fill (J) circle (0pt) node[below] {\small$\Omega_7$};
		\coordinate (k) at (2.2,-0.2);
		\fill (k) circle (0pt) node[right] {\small$\Omega_8$};
		\fill (I) circle (0pt) node[below] {$0$};
		\draw[red,  thick] (2,0) arc (0:360:2);
		\draw[blue] (2,3) circle (0.12);
		\draw[blue][->](0,0)--(-1.5,0);
		\draw[blue][->](-1.5,0)--(-2.8,0);
		\draw[blue][->](0,0)--(1.5,0);
		\draw[blue][->](1.5,0)--(2.8,0);
		\draw[blue][->](0,2.7)--(0,2.2);
		\draw[blue][->](0,1.6)--(0,0.8);
		\draw[blue][->](0,-2.7)--(0,-2.2);
		\draw[blue][->](0,-1.6)--(0,-0.8);
		\coordinate (A) at (2,3);
		\coordinate (B) at (2,-3);
		\coordinate (C) at (-0.5546996232,0.8320505887);
		\coordinate (D) at (-0.5546996232,-0.8320505887);
		\coordinate (E) at (0.5546996232,0.8320505887);
		\coordinate (F) at (0.5546996232,-0.8320505887);
		\coordinate (G) at (-2,3);
		\coordinate (H) at (-2,-3);
		\coordinate (I) at (2,0);
		\draw[blue] (2,-3) circle (0.12);
		\draw[blue] (-0.55469962326,0.8320505887) circle (0.12);
		\draw[blue] (0.5546996232,0.8320505887) circle (0.12);
		\draw[blue] (-0.5546996232,-0.8320505887) circle (0.12);
		\draw[blue] (0.5546996232,-0.8320505887) circle (0.12);
		\draw[blue] (-2,3) circle (0.12);
		\draw[blue] (-2,-3) circle (0.12);
		\coordinate (J) at (1.7320508075688774,1);
		\coordinate (K) at (1.7320508075688774,-1);
		\coordinate (L) at (-1.7320508075688774,1);
		\coordinate (M) at (-1.7320508075688774,-1);
		\fill (A) circle (1pt) node[right] {$z_n$};
		\fill (B) circle (1pt) node[right] {$\overline{z}_n$};
		\fill (C) circle (1pt) node[left] {$-\frac{1}{z_n}$};
		\fill (D) circle (1pt) node[left] {$-\frac{1}{\overline{z}_n}$};
		\fill (E) circle (1pt) node[right] {$\frac{1}{\overline{z}_n}$};
		\fill (F) circle (1pt) node[right] {$\frac{1}{z_n}$};
		\fill (G) circle (1pt) node[left] {$-\overline{z}_n$};
		\fill (H) circle (1pt) node[left] {$-z_n$};
		\fill (I) circle (1pt) node[above] {$1$};
		\fill (J) circle (1pt) node[right] {$\zeta_m$};
		\fill (K) circle (1pt) node[right] {$\overline{\zeta}_m$};
		\fill (L) circle (1pt) node[left] {$-\overline{\zeta}_m$};
		\fill (M) circle (1pt) node[left] {$-\zeta_m$};
		\end{tikzpicture}}
\centerline{\noindent {\small \textbf{Figure 5.} The gray region is $\Omega$. The blue circle constitute $\Sigma^{(2)}$ together. }}

\begin{prop}\label{pp2}
Let $\xi\in(5,7)$, and $G(s)=s+s^{-1}$ is a real-valued function. For $z=re^{i\psi}$, then the phase function satisfies the following estimation
\begin{align}
&Re(2it\theta)\leq\frac{t}{8}|\sin2\psi|(|\xi-6|+3)G^{2}(r),\quad z\in\Omega_{j},\quad j=1,3,5,7,\label{est1}\\
&Re(2it\theta)\geq-\frac{t}{8}|\sin2\psi|(|\xi-6|+3)G^{2}(r),\quad z\in\Omega_{j},\quad j=2,4,6,8.
\end{align}
\end{prop}
\begin{proof}
Taking $\Omega_1$ as an example, the other regions can be proved similarly.  It follows that from \eqref{Re}
\begin{align*}
Re(2it\theta)=tImzRez\left[(\xi-6)(1+|z|^{-4})+(Rez^2-Imz^2)(1+|z|^{-4})\right].
\end{align*}
Let $z=re^{i\psi}$, then one has
\begin{align}
Re(2it\theta)&=\frac{t}{2}r^2\sin2\psi\left[(\xi-6)(1+r^{-4})+r^2\cos2\psi(1+r^{-8})\right]\nonumber\\
&=\frac{t}{2}\sin2\psi\left[(\xi-6)(r^2+r^{-2})+\cos2\psi(r^4+r^{-4})\right].\label{Ret}
\end{align}
Additionally, from $G(r)=r^2+r^{-2}$ and  $G^2(r)=r^4+r^{-4}+2$ the above expression is further reduced to
\begin{align*}
Re(2it\theta)=\frac{t}{2}\sin2\psi\left[(\xi-6)G(r)+(G^2(r)-2)\cos2\psi\right].
\end{align*}
Observing that for $z\in\Omega_1$, $\frac{2|\xi-6|}{|\xi-6|+1}<\cos2\psi$, we then have
\begin{align*}
\frac{|\xi-6|}{\cos2\psi}G(r)\leq\frac{|\xi-6|+1}{4}G^2(r).
\end{align*}
Substituting the above formula into \eqref{Ret} yields the estimate \eqref{est1}.
\end{proof}

The estimation of proposition \ref{pp2} is the basis of continuous extension of triangular decomposition of jump matrix \eqref{jumpv1}. To do so, we need to extend the matrices $U_L,U_R$ and $W_L,W_R$  off the interval $\Sigma$, which is the content of the following proposition. Additionally, a sufficiently small positive constant $\upsilon_0$ is introduced to satisfy the equality $(1-\upsilon_0)\cos\psi>\frac{1}{2}$. Let $\mathcal {X}_1$ denotes the characterization function on the interval $(1-\upsilon_0,1+\upsilon_0)$ and $\mathcal {X}_0$ is the  characterization function  supporting in $(-\upsilon_0,\upsilon_0)$ with $\mathcal {X}_0(z)=\mathcal {X}_1(1+z)$. Now for brevity, define
\begin{align}
&\mathcal {H}_1=\mathcal {H}_5=\rho(z),\quad\qquad\quad~~ \mathcal {H}_4=\mathcal {H}_8=\widetilde{\rho}(z),\\
&\mathcal {H}_2=\mathcal {H}_6=\frac{\widetilde{\rho}(z)}{1-\widetilde{\rho}(z)\rho(z)},\quad
\mathcal {H}_3=\mathcal {H}_7=\frac{\rho(z)}{1-\widetilde{\rho}(z)\rho(z)},
\end{align}
and
\begin{align}
R^{(2)}(z)=\left\{\begin{array}{lll}
\left(\begin{array}{cc}
1 & R_j(z)e^{2it\theta}\\
0 & 1
\end{array}\right), & z\in \Omega_j,~~j=2,4,6,8,\\
\\
\left(\begin{array}{cc}
1 & 0\\
R_j(z)e^{-2it\theta} & 1
\end{array}\right),  &z\in \Omega_j,~~j=1,3,5,7,\\
\\
I,  &elsewhere.\\
\end{array}\right.\label{R2}
\end{align}

\begin{prop}\label{pp4}
$R_j$: $\overline{\Omega}_j\to C$, $j=1,2,..,8$ have boundary values as follow:
\begin{align}
&R_1(z)=\Bigg\{\begin{array}{ll}
-\rho(z)T(z)^{2} & z\in \mathbb{R}^+,\\
0  &z\in \Sigma_1,\\
\end{array} \hspace{0.6cm}
R_2(z)=\Bigg\{\begin{array}{ll}
0  &z\in \Sigma_2,\\
-\dfrac{\widetilde{\rho}(z) T_-(z)^{-2}}{1-\rho(z)\widetilde{\rho}(z)} &z\in  i\mathbb{R}^+,\\
\end{array} \\
&R_3(z)=\Bigg\{\begin{array}{ll}
\dfrac{\rho(z) T_+(z)^2}{1-\rho(z)\widetilde{\rho}(z)} &z\in i\mathbb{R}^+, \\
0 &z\in \Sigma_3,\\
\end{array}\hspace{0.6cm}
R_4(z)=\Bigg\{\begin{array}{ll}
0  &z\in \Sigma_4,\\
-\widetilde{\rho}(z)T(z)^{-2} &z\in  \mathbb{R}^-,\\
\end{array} \\
&R_5(z)=\Bigg\{\begin{array}{ll}
-\rho(z)T(z)^{2} &z\in  \mathbb{R}^-,\\
0  &z\in \Sigma_5,
\end{array} \hspace{0.5cm}
R_6(z)=\Bigg\{\begin{array}{ll}
0  &z\in \Sigma_6,\\
-\dfrac{\widetilde{\rho}(z) T_-(z)^{-2}}{1-\rho(z)\widetilde{\rho}(z)} &z\in  i\mathbb{R}^-,\\
\end{array} \\
&R_7(z)=\Bigg\{\begin{array}{ll}
\dfrac{\rho(z) T_+(z)^2}{1-\rho(z)\widetilde{\rho}(z)} &z\in i\mathbb{R}^-, \\
0  &z\in \Sigma_7,
\end{array} \hspace{0.5cm}
R_8(z)=\Bigg\{\begin{array}{ll}
0  &z\in \Sigma_8,\\
-\widetilde{\rho}(z)T(z)^{-2} &z\in  \mathbb{R}^+.\\
\end{array}
\end{align}	
For $z\in\Omega_k$, the matrices $R_k$ are of the following estimation
\begin{align}
\left|\overline{\partial}R_j(z)\right|\lesssim\left\{\begin{array}{ll}
|\mathcal {H}_j'(|z|)|+|z|^{-1/2}, & z\in \Omega_k,~~k=1,5,4,8,\label{dbarRj}\\
\\
|{H}_j'(i|z|)|+|z|^{-1/2}+|\overline{\partial}\mathcal {X}_1(|z|)|,  &z\in \Omega_k,~~k=2,3,6,7.\\
\end{array}\right.
\end{align}
Also in the small neighborhood of the singularities $\pm i$, there is the following estimate
\begin{align}
&|\overline{\partial}R_j(z)|\lesssim |z\mp i|,~~ z\in \Omega_k,~~k=2,3,6,7, \label{Ri}
\end{align}
and for $z\in$ elsewhere
\begin{align}
\overline{\partial}R_j(z)=0.
\end{align}
\end{prop}

\begin{proof}
Unlike the initial values of the Schwarz space type presented in \cite{fNLS}, the proof is divided into two parts, one with no singularities and the other with singularities, since the determinant of the matrix $M(z)$ has singularities $\pm i$.

For case (i): No singularities as $z\in \overline{\Omega}_j$, $j=1,4,5,8$. Take $R_1(z)$ only as an example, the rest of the cases are similar.

Similar to the proposition 2.1 shown in \cite{McLaughlin-3},  the matrix $R_1(z)$ can be expressed as
\begin{align}
R_1(z)=\mathcal {H}_1(|z|)T^2(z)\cos(z_0 \arg z),\hspace{0.5cm}z_0=\frac{2\pi}{\psi}.
\end{align}
where the function $\mathcal {H}_1(|z|)$ is bounded, from which we obtain the $\overline{\partial}$ derivatives
\begin{align}\label{est-R1}
\overline{\partial}R_1(z)=\frac{e^{i\tau}}{2}T^2(z)\left(\mathcal {H}_1'(r)\cos(z_0\tau)-\frac{i}{r}\mathcal {H}_1(r)z_0\sin(z_0\tau) \right),
\end{align}
where we denote $r$ as the mode of $z$ with $z=re^{i\tau}$ and the relationship $\overline{\partial}=\frac{e^{i\tau}}{2}\left(\partial_r+\frac{i}{r} \partial_\tau\right) $ is derived.
It is observed that the matrix function $T(z)$ is bounded in region $\overline{\Omega}_1$. Therefore, in order to estimate the above equation it is necessary to further consider the  second term of the left-hand side of equation \eqref{est-R1}. Employing the Cauchy-Schwarz inequality yields
\begin{align}
|\mathcal {H}_1(r)|= |\rho(r)|= |\rho(r)-\rho(0)|=\left|\int_{0}^r\rho'(s)ds\right|\leq \parallel \rho'(s)\parallel_{L^2} r^{1/2}.
\end{align}
Then the boundedness of \eqref{dbarRj}  follows immediately.

For case (ii): Existence of singularity as $z\in \overline{\Omega}_j$, $j=2,3,6,7.$ Again, only one case $R_2(z)$ is considered, and the rest are proved using the same method.

As observed on \eqref{s11} and \eqref{s22}, the scattering data $s_{11}(z)$ and $s_{22}(z)$ are of singularity at $\pm i$. This suggest that the matrix $R_2(z)$ also has singularities. Therefore, in the estimation of $R_2(z)$, it is necessary to control the behavior at the singularity. Fix a small constant $\delta_0$ stratifying $\varphi>\delta_0\epsilon_0$, the function $R_2(z)$ in $\Omega_2$ can be written as
\begin{align*}
R_2(z)=-\dfrac{\widetilde{\rho}(z)}{1-\rho(z)\widetilde{\rho}(z)} T_-(z)^{-2}=\frac{\overline{s_{21}}(z)}{s_{11}(z)}
\left(\frac{s_{11}(z)}{T_{-}(z)}\right)^2,
\end{align*}
further we denote
\begin{align}
R_2(z)=R_{21}(z)+R_{22}(z),
\end{align}
with
\begin{align*}
R_{21}(z)&=[1-\mathcal {X}_1(|z|)]\mathcal {H}_2(i|z|)T^{-2}(z)\cos[k_0(\frac{\pi}{2}-\arg z)],\\
R_{22}(z)&=f(|z|)g(z)\cos[z_0(\frac{\pi}{2}-\arg z)]
-\frac{i|z|}{z_0}\mathcal {X}_0(\frac{\arg z}{\delta_0})f'(|z|)g(z)\sin[z_0(\frac{\pi}{2}-\arg z)],
\end{align*}	
where $f'(s)$ is the derivative of the $f(s)$ and
\begin{align}
f(z)=\mathcal {X}_1(z)\frac{\overline{s_{21}}(z)}{s_{11}(z)},\hspace{0.5cm}g(z)=\left( \frac{s_{11}(z)}{T(z)}\right) ^2.
\end{align}

In what follows, we will estimate the $\overline{\partial}$ derivatives of the expressions $R_{21}(z)$ and $R_{22}(z)$. Consider  $R_{21}(z)$ firstly, similar to the case (i), let $z=re^{i\phi}$, then one has
\begin{align}
|\overline{\partial}R_{21}(z)|\lesssim (1-\mathcal {X}_1(r))\left( |\mathcal {H}_2'(ir)|+r^{-1/2}\right)+|\overline{\partial}\mathcal {X}_1(r)|,
\end{align}
by using the fact $R_{21}(z)\equiv0$ in support of $\mathcal {X}_1$, $R_{22}(z)\equiv0$ out support of $\mathcal {X}_1$ and  for $z$ out of supp$(X_1)$
\begin{align*}
|\mathcal {H}_2(z)|=\left|\dfrac{\widetilde{\rho}(z) }{1-\rho(z)\widetilde{\rho}(z)}\right|=|\dfrac{\widetilde{\rho}(z) }{1-|\rho(z)|^2}|\lesssim |\rho(z)|.
\end{align*}

Now we turn to $\overline{\partial}R_{22}(z)$. Let $z=re^{i\phi}$ as before, and $\overline{\partial}=\frac{e^{i\phi}}{2}
\left(\partial_r+\frac{i}{r}\partial_\phi\right)$, Then
\begin{align}
\overline{\partial}R_{22}(z)=&\frac{e^{i\phi}}{2}g(z) \cos[z_0(\frac{\pi}{2}-\phi)]f'(ir)\left( 1-\mathcal {X}_0(\frac{\phi}{\delta_0})\right)-\frac{i}{z_0}\sin[z_0(\frac{\pi}{2}-\phi)] \nonumber\\
&\mathcal {X}_0(\frac{\arg z}{\delta_0})(rf'(ir))'
+ \sin[z_0(\frac{\pi}{2}-\phi)]\left[\frac{iz_0}{r}f(ir)+\frac{1}{\delta_0 z_0}\mathcal {X}_0'(\frac{\phi}{\delta_0})f'(ir) \right].
\end{align}
Similar to the lemma 6.5 shown in \cite{SandRNLS}, the three terms on the right side of the equation can be estimated separately to obtain the estimation of $|\overline{\partial}R_{22}(z)|$, therefore the results \eqref{dbarRj} is proved.
Finally for $z\sim i$,
\begin{equation}
|\overline{\partial}R_{22}(z)|\lesssim \left|\sin[z_0(\frac{\pi}{2}-\phi)]\right|+\left|1-\mathcal {X}_0(\phi/\delta_0)\right|=\mathcal{O}(\phi),
\end{equation}
from which \eqref{Ri}  follows immediately.
\end{proof}
Additionally, the matrix $R^{(2)}(z)$ satisfies symmetry
\begin{align} R^{(2)}(z)=\sigma_2\overline{R^{(2)}(\bar{z})}\sigma_2=\sigma_1\overline{R^{(2)}(-\bar{z})}\sigma_1
=\sigma_3Q_-R^{(2)}(-1/z)\sigma_3Q_-.
\end{align}

A new matrix-valued function $M^{(2)}(z)=M^{(2)}(x,t;z)$ with no jumps on the contour $\Sigma$ is introduced via the matrix $R^{(2)}(z)$
\begin{align}\label{M2}
M^{(2)}(z)=M^{(1)}(z)R^{(2)}(z),
\end{align}
which follows the RH problem \ref{RHP3}.
\begin{RHP}\label{RHP3}
Find a $2\times2$ matrix valued function $M^{(2)}(z)$, which is continuous in $\mathbb{C}$, sectionally continuous first partial derivatives in $\mathbb{C}\backslash\Sigma^{(2)}\cup\{\eta_n,\overline{\eta}_n\}_{n\in\Lambda}$ satisfying that
\begin{enumerate}[(1)]
\item The symmetry:
\begin{align*}
M^{(2)}(x,t;z)=\sigma_2\overline{M^{(2)}(x,t;\overline{z})}\sigma_2=\sigma_1\overline{M^{(2)}(x,t;-\overline{z})}\sigma_1
    =\frac{i}{z}M^{(2)}(x,t;-1/z)\sigma_3Q_-.
\end{align*}
\item Take the continuous boundary values $M^{(2)}_+(x,t;z)$ $($respectively $M^{(2)}_-(x,t;z)$$)$ on $\Sigma^{(2)}$ from left (respectively right), then
\begin{align}
M^{(2)}_+(x,t;z)=M^{(2)}_-(x,t;z)V^{(2)}(z),\quad z\in\Sigma^{(2)},
\end{align}
where
\begin{align}\label{jumpv2}
	V^{(2)}(z)=\left\{ \begin{array}{ll}
		\left(\begin{array}{cc}
			1 & 0\\
			-C_n(z-\eta_n)^{-1}T^2(z)e^{-2it\theta_n} & 1
		\end{array}\right),   &\text{as } 	|z-\eta_n|=\epsilon_0,~~n\in\nabla\setminus\Lambda,\\[12pt]
		\left(\begin{array}{cc}
			1 & -C_n^{-1}(z-\eta_n)T^{-2}(z)e^{2it\theta_n}\\
			0 & 1
		\end{array}\right),   &\text{as } |z-\eta_n|=\epsilon_0,~~n\in\Delta,\\
		\left(\begin{array}{cc}
			1 & \overline{C}_n(z-\overline{\eta}_n)^{-1}T^{-2}(z)e^{2it\overline{\theta}_n}\\
			0 & 1
		\end{array}\right),   &\text{as } 	|z-\overline{\eta}_n|=\epsilon_0,~~n\in\nabla\setminus\Lambda,\\
		\left(\begin{array}{cc}
			1 & 0	\\
			\overline{C}_n^{-1}(z-\overline{\eta}_n)e^{-2it\bar{\theta}_n}T^2(z) & 1
		\end{array}\right),   &\text{as } |z-\overline{\eta}_n|=\epsilon_0,~~n\in\Delta.\\
	\end{array}\right.
\end{align}
\item Residue conditions: $M^{(2)}(x,t;z)$ has simple poles at each point $\eta_n$ and $\overline{\eta}_n$ for $n\in\Lambda$ with:
\begin{align}
	&\res_{z=\eta_n}M^{(2)}(x,t;z)=\lim_{z\to \eta_n}M^{(2)}(x,t;z)\left(\begin{array}{cc}
		0 & 0\\
		C_ne^{-2it\theta_n}T^2(\eta_n) & 0
	\end{array}\right),\\
	&\res_{z=\overline{\eta}_n}M^{(2)}(x,t;z)=\lim_{z\to \overline{\eta}_n}M^{(2)}(x,t;z)\left(\begin{array}{cc}
		0 & -\overline{C}_nT^{-2}(\overline{\eta}_n)e^{2it\overline{\theta}_n}\\
		0 & 0
	\end{array}\right).
\end{align}
\item The asymptotic behavior
\begin{align*}
&M^{(2)}(x,t;z)=e^{i\nu_{-}(x,q)\sigma_{3}}+\mathcal {O}(z^{-1}),\quad\qquad z\rightarrow\infty,\\
&M^{(2)}(x,t;z)=\frac{1}{z}e^{i\nu_{-}(x,q)\sigma_{3}}\sigma_{3}Q_{-}+\mathcal {O}(1),\quad z\rightarrow0.
\end{align*}
\item $\overline{\partial}$-derivative: for $z\in\mathbb{C}$
\begin{align}\label{Wz}
\overline{\partial}M^{(2)}(x,t;z)=M^{(2)}(x,t;z)W(z),
\end{align}
where
\begin{equation}
W(z)=\left\{\begin{array}{lll}
\left(\begin{array}{cc}
0 & \overline{\partial}R_j(z)e^{2it\theta}\\
0 & 0
\end{array}\right), & z\in \Omega_j,~~j=1,3,5,7,\\
\\
\left(\begin{array}{cc}
0 & 0\\
\overline{\partial}R_j(z)e^{-2it\theta} & 0
\end{array}\right),  &z\in \Omega_j,~~j=2,4,6,8,\\
\\
0,  &elsewhere.\\
\end{array}\right. \label{DBARR2}
\end{equation}
\end{enumerate}
\end{RHP}
\subsection{The $\mathcal {N}$-soliton solutions}\quad
The next step is to decompose the mixed ($\overline{\partial}$-) RH problem \ref{RHP3}, that is, remove the soliton component  corresponding to $M^{(sol)}(x,t;z)$ from the RH problem \ref{RHP3}, and a pure $\overline{\partial}$-problem with non-zero $\overline{\partial}$-derivative, which can be analyzed by using the known $\overline{\partial}$-technique. Consider the pure RH problem satisfied by the $M^{(sol)}(x,t;z)$ with the $W(z)\equiv0$.
\begin{RHP}\label{RHP4}
Find a $2\times2$ matrix valued function $M^{(sol)}(x,t;z)$, which is meromorphic function in $\mathbb{C}\backslash\Sigma^{(2)}$ satisfying that
\begin{enumerate}[(1)]
\item The symmetry:
\begin{align*}
M^{(sol)}(x,t;z)=\sigma_2\overline{M^{(sol)}(x,t;\overline{z})}\sigma_2=\sigma_1
\overline{M^{(sol)}(x,t;-\overline{z})}\sigma_1
    =\frac{i}{z}M^{(sol)}(x,t;-1/z)\sigma_3Q_-.
\end{align*}
\item Take the continuous boundary values $M^{(sol)}_+(x,t;z)$ $($respectively $M^{(sol)}_-(x,t;z))$ on $\Sigma^{(2)}$ from left $($respectively right$)$, then
\begin{align}
M^{(sol)}_+(x,t;z)=M^{(sol)}_-(x,t;z)V^{(2)}(z),\quad z\in\Sigma^{(2)},
\end{align}
where
\begin{align}
	V^{(2)}(z)=\left\{ \begin{array}{ll}
		\left(\begin{array}{cc}
			1 & 0\\
			-C_n(z-\eta_n)^{-1}T^2(z)e^{-2it\theta_n} & 1
		\end{array}\right),   &\text{as } 	|z-\eta_n|=\epsilon_0,~~n\in\nabla\setminus\Lambda,\\[12pt]
		\left(\begin{array}{cc}
			1 & -C_n^{-1}(z-\eta_n)T^{-2}(z)e^{2it\theta_n}\\
			0 & 1
		\end{array}\right),   &\text{as } |z-\eta_n|=\epsilon_0,~~n\in\Delta,\\
		\left(\begin{array}{cc}
			1 & \overline{C}_n(z-\overline{\eta}_n)^{-1}T^{-2}(z)e^{2it\bar{\theta}_n}\\
			0 & 1
		\end{array}\right),   &\text{as } 	|z-\overline{\eta}_n|=\epsilon_0,~~n\in\nabla\setminus\Lambda,\\
		\left(\begin{array}{cc}
			1 & 0	\\
			\overline{C}_n^{-1}(z-\overline{\eta}_n)e^{-2it\bar{\theta}_n}T^2(z) & 1
		\end{array}\right),   &\text{as } 	|z-\overline{\eta}_n|=\epsilon_0,~~n\in\Delta.\\
	\end{array}\right.
\end{align}
\item Residue conditions: $M^{(sol)}(x,t;z)$ has simple poles at each point $\eta_n$ and $\overline{\eta}_n$ for $n\in\Lambda$ with:
\begin{align}
	&\res_{z=\eta_n}M^{(sol)}(x,t;z)=\lim_{z\to \eta_n}M^{(sol)}(x,t;z)\left(\begin{array}{cc}
		0 & 0\\
		C_ne^{-2it\theta_n}T^2(\eta_n) & 0
	\end{array}\right),\\
	&\res_{z=\overline{\eta}_n}M^{(sol)}(x,t;z)=\lim_{z\to \overline{\eta}_n}M^{(sol)}(x,t;z)\left(\begin{array}{cc}
		0 & -\overline{C}_nT^{-2}(\overline{\eta}_n)e^{2it\overline{\theta}_n}\\
		0 & 0
	\end{array}\right).
\end{align}
\item
$M^{(sol)}(x,t;z)=e^{i\nu_{-}(x,q)\sigma_{3}}+\mathcal {O}(z^{-1}),\quad\qquad z\rightarrow\infty,$\\
$M^{(sol)}(x,t;z)=\frac{1}{z}e^{i\nu_{-}(x,q)\sigma_{3}}\sigma_{3}Q_{-}+\mathcal {O}(1),\quad z\rightarrow0.$
\item $\bar{\partial}$-Derivative:  $W(z)=0$, for $ z\in \mathbb{C}$.
\end{enumerate}
\end{RHP}

The following proposition \ref{pp-deng} will show the correlation between the existence and uniqueness of the solutions of the RH problem \ref{RHP4} and the original RH problem \ref{RHP1}.
\begin{prop}\label{pp-deng}
For the scattering data $\left\lbrace  r(z),\left\lbrace \eta_n,C_n\right\rbrace_{n\in\Lambda}\right\rbrace$ corresponding to the RH problem \ref{RHP3}, there exists a unique solution $M^{(sol)}(z)$ of RH problem \ref{RHP4} $($corresponding to the RH problem \ref{RHP3} with $W(z)\equiv0$$)$, and it is equivalent, by a transformation, to the reflection-less solution of the original RH problem \ref{RHP1} with the modified scattering data
$\left\lbrace  0,\left\lbrace \eta_n,\mathring{c}_n\right\rbrace_{n\in\Lambda}\right\rbrace$, where
\begin{align}
\mathring{c}_n(x,t)=C_n\exp\left\lbrace-\frac{1}{i\pi}\int_{\mathbb{R}}\log(1-|\rho(s)|^2)
\left(\frac{1}{s-\eta_n}-\frac{1}{2s} \right)  \right\rbrace,
\end{align}
where $\rho(z)$ is reflection coefficient defined by \eqref{reflection}.
\end{prop}
\begin{proof}
In order to match the $M^{(sol)}(z)$ and the solution of the original RH problem \ref{RHP1}, the jump conditions and poles need to be restored. It follows from \eqref{M1} and \eqref{M2} that
\begin{align}\label{Mj}
M^{(\sharp)}(x,t;z)=\left(\prod_{n\in \Delta}\eta_n \right)^{-\sigma_3} M^{(sol)}(z)T(z)^{-\widehat{\sigma}_3}F^{-1}(z)\left( \prod_{n\in \Delta}\dfrac{z-\eta_n}{\overline{\eta}_n^{-1}z-1}\right) ^{-\sigma_3},
\end{align}
where
\begin{align}
F(z)=\left\{ \begin{array}{ll}
		\left(\begin{array}{cc}
			1 & 0\\
			-C_n(z-\eta_n)^{-1}e^{-2it\theta_n} & 1
		\end{array}\right),   &\text{as } |z-\eta_n|=\epsilon_0,~~n\in\nabla\setminus\Lambda,\\[12pt]
		\left(\begin{array}{cc}
			1 & -C_n^{-1}(z-\eta_n)e^{2it\theta_n}\\
			0 & 1
		\end{array}\right),   &\text{as } |z-\eta_n|=\epsilon_0,~~n\in\Delta,\\
		\left(\begin{array}{cc}
		1 & \overline{C}_n(z-\overline{\eta}_n)^{-1}e^{2it\overline{\theta}_n}\\
		0 & 1
		\end{array}\right),   &\text{as } 	|z-\overline{\eta}_n|=\epsilon_0,~~n\in\nabla\setminus\Lambda,\\
		\left(\begin{array}{cc}
		1 & 0	\\
		\overline{C}_n^{-1}(z-\overline{\eta}_n)e^{-2it\overline{\theta}_n} & 1
		\end{array}\right),   &\text{as } 	|z-\overline{\eta}_n|=\epsilon_0,~~n\in\Delta.
	\end{array}\right.
\end{align}
It's clear that the transformation $M^{(\sharp)}(x,t;z)$ ensures the normalization conditions at origin and infinity. By comparing \eqref{jumpv2} and \eqref{Mj}, we find that $M^{(\sharp)}(x,t;z)$  has no jump conditions. It follows from \eqref{M1} that for $n\notin\Lambda$ the transformation \eqref{Mj} restores the jump conditions on disks $|z-\eta_n|=\epsilon_0$ and $|z-\overline{\eta}_n|=\epsilon_0$. The following will show that $M^{(\sharp)}(x,t;z)$ satisfies the residue condition of the original RH problem \ref{RHP1}, where parameter $C_n$ is replaced by $\mathring{c}_n$. Taking $\eta_n$ as an example for $n\in\Lambda$, another case is similarly provable.
\begin{align}
\res_{z=\eta_n}&M^{(\sharp)}(x,t;z)=\left(\prod_{n\in \Delta}\eta_n \right)^{-\sigma_3}\res_{z=\eta_n}M^{(sol)}(x,t;z)T^{-\widehat{\sigma}_3}G(z)^{-1}\left( \prod_{n\in \Delta}\dfrac{z-\eta_n}{\overline{\eta}_n^{-1}z-1}\right) ^{-\sigma_3}\nonumber\\
	=&\lim_{z\to \eta_n}-\left(\prod_{n\in \Delta}\eta_n \right)^{-\sigma_3}M^{(sol)}(x,t;z)\left(\begin{array}{cc}
		0 & 0\\
		C_ne^{-2it\theta_n}T^2(\eta_n) & 0
\end{array}\right)\left( \prod_{n\in \Delta}\dfrac{z-\eta_n}{\overline{\eta}_n^{-1}z-1}\right) ^{-\sigma_3}\nonumber\\
		=&\lim_{z\to \eta_n}M^{(\sharp)}(x,t;z)\left(\begin{array}{cc}
			0 & 0\\
			\mathring{c}_ne^{-2it\theta_n} & 0
\end{array}\right).\nonumber
\end{align}
The analyticity and symmetry of function $M^{(\sharp)}(x,t;z)$ can be deduced from the related properties of $M^{(sol)}(x,t;z)$, $T(z)$ and $F(z)$. In conclusion, $M^{(sol)}(x,t;z)$ is the solution of the original RH problem \eqref{RHP1} without reflection. Additionally, the existence and uniqueness of the solution can be similar to that in appendix A of reference \cite{SandRNLS}. Finally, we complete the proof of the proposition.
\end{proof}

Although $M^{(sol)}(x,t;z)$ exists and is unique, it can not give an exact expression.  In what follows, we use a new solution $M_{\Lambda}^{(sol)}(x,t;z)$ to approximate $M^{(sol)}(x,t;z)$, and the resulting error can be characterized by the small norm theory of RH problem \cite{small-1,small-2}. Observing that the jump matrix $V^{(2)}(z)$ generated by the RH problem \ref{RHP4} such that the following estimation
\begin{align*}
\left|\left|V^{(2)}(z)-I\right|\right|_{L^{2}(\Sigma^{(2)})}=\left|C_{n}(z-\zeta_n)^{-1}T(z)^{2}e^{-2it\theta_n}
\right|\lesssim e^{Re(2it\theta)}\lesssim \mathcal {O}(e^{-2\varrho_0t}),
\end{align*}
which gives rise to that for $1\leq p\leq+\infty$, the jump matrix $V^{(2)}(z)$ admits
\begin{align}\label{est-V2}
\left|\left|V^{(2)}(z)-I\right|\right|_{L^{2}(\Sigma^{(2)})}\leq K_pe^{-2\varrho_0t},
\end{align}
where $K_p$ is a positive constant and dependent on $p$. The expression \eqref{est-V2} implies that the  jump matrix $V^{(2)}(z)$ converges uniformly to the identity matrix, which means that the contribution of the jump matrix to the long-time asymptotic behavior of the solution is meaningless. On the contrary, the dominant contribution to the long-time asymptotic behavior of solutions comes from the simple poles of $M^{(sol)}(x,t;z)$, i.e., the simple poles $\zeta_n$ for $n\in\Lambda$.  Based on this idea, consider the following decomposition
\begin{align}\label{ML}
M^{(sol)}(x,t;z)=M^{(err)}(z)M_\Lambda^{(sol)}(x,t;z),
\end{align}
where $M^{(err)}(z)$ denotes the error function, which is determined by the small norm theory of RH problem, and $M_\Lambda^{(sol)}(x,t;z)$ is the solution of  the RH problem \ref{RHP4} with $V^{(2)}(z)\equiv0$.
\begin{RHP}\label{RHP5}
Find a matrix-valued function $M_\Lambda^{(sol)}(x,t;z)$ satisfies
\begin{enumerate}[(1)]
\item The analyticity: $M_\Lambda^{(sol)}(x,t;z)$ is analytic in $\mathbb{C}\backslash\{\eta_n,\overline{\eta}_n\}$.
\item The symmetries:
\begin{align}\label{SymML}
M_\Lambda^{(sol)}(x,t;z)=\sigma_2\overline{M^{(sol)}_\Lambda(x,t;\overline{z})}\sigma_2=\sigma_1
\overline{M^{(sol)}_\Lambda(x,t;-\overline{z})}\sigma_1=
\frac{i}{z}M^{(sol)}_\Lambda(x,t;-1/z)\sigma_3Q_-.
\end{align}
\item Residue conditions:  $M_\Lambda^{(sol)}(x,t;z)$ has simple poles at each point $\eta_n$ and $\overline{\eta}_n$ for $n\in\Lambda$ with:
\begin{align}
	&\res_{z=\eta_n}M_\Lambda^{(sol)}(x,t;z)=\lim_{z\to \eta_n}M_\Lambda^{(sol)}(x,t;z)\left(\begin{array}{cc}
		0 & 0\\
		C_ne^{-2it\theta_n}T^2(\eta_n) & 0
	\end{array}\right),\\
	&\res_{z=\overline{\eta}_n}M_\Lambda^{(sol)}(x,t;z)=\lim_{z\to \overline{\eta}_n}M_\Lambda^{(sol)}(x,t;z)\left(\begin{array}{cc}
		0 & -\overline{C}_nT^{-2}(\overline{\eta}_n)e^{2it\overline{\theta}_n}\\
		0 & 0
	\end{array}\right).
\end{align}
\item The asymptotic behavior:
\begin{align*}
&M_\Lambda^{(sol)}(x,t;z)=e^{i\nu_{-}(x,t;q)\sigma_{3}}+\mathcal {O}(z^{-1}),\quad\qquad z\rightarrow\infty,\\
&M_\Lambda^{(sol)}(x,t;z)=\frac{1}{z}e^{i\nu_{-}(x,t;q)\sigma_{3}}\sigma_{3}Q_{-}+\mathcal {O}(1),\quad z\rightarrow0.
\end{align*}
\end{enumerate}
\end{RHP}

For convenience, we omit the independent variables $x$ and $t$,  denoting $M_\Lambda^{(sol)}(x,t;z)=M_\Lambda^{(sol)}(z)$. It follows any solution satisfying RH problem \ref{RHP5} from the symmetry \eqref{SymML} must have the following partial fractional expansion for $\Lambda\neq\emptyset$
\begin{align}\label{msol2}
M^{(sol)}_\Lambda(z)&=e^{i\nu_-(q^{(sol)}_\Lambda)\sigma_3}+\frac{i}{z}e^{i\nu_-(q^{(sol)}_\Lambda)\sigma_3}\sigma_3Q_-\nonumber\\
		&+\sum_{s=1}^{n_2}\left[\left(\begin{array}{cc}
			\frac{\alpha_s}{z-w_{i_s}} & \frac{\overline{\tau_s}}{z-\overline{w}_{i_s}}\\
			\frac{\tau_s}{z-w_{i_s}} & \frac{\overline{\alpha_s}}{z-\overline{w}_{i_s}}
		\end{array}\right)+\left(\begin{array}{cc}
		-\frac{\alpha_s}{z+w_{i_s}} & \frac{\overline{\tau_s}}{z+\overline{w}_{i_s}}\\
		\frac{\tau_s}{z+w_{i_s}} & -\frac{\overline{\alpha_s}}{z+\overline{w}_{i_s}}
	\end{array}\right) \right]\nonumber\\
&+\sum_{k=1}^{n_1}\left[\left(\begin{array}{cc}
	\frac{\beta_k}{z-z_{j_k}} & \frac{\overline{\varsigma_k}}{z-\overline{z}_{j_k}}\\
	\frac{\varsigma_k}{z-z_{j_k}} & \frac{\overline{\beta_k}}{z-\overline{z}_{j_k}}
\end{array}\right)+\left(\begin{array}{cc}
	-\frac{\beta_k}{z+z_{j_k}} & \frac{\overline{\varsigma_k}}{z+\overline{z}_{j_k}}\\
	\frac{\varsigma_k}{z+z_{j_k}} & -\frac{\overline{\beta_k}}{z+\overline{z}_{j_k}}
\end{array}\right) \right]\nonumber\\
&+\sum_{k=1}^{n_1}i\left[\left(\begin{array}{cc}
	\frac{-\overline{q_-\beta_k}}{\overline{z}_{j_k}z-1} & \frac{-q_-\varsigma_k}{z_{j_k}z-1}\\
	 \frac{-\overline{q_-\varsigma_k}}{\overline{z}_{j_k}z-1} & \frac{q_-\beta_k}{z_{j_k}z-1}
\end{array}\right)+\left(\begin{array}{cc}
	\frac{\overline{q_-\beta_k}}{\overline{z}_{j_k}z+1} & \frac{-q_-\varsigma_k}{z_{j_k}z+1}\\
	\frac{-\overline{q_-\varsigma_k}}{\overline{z}_{j_k}z+1} & \frac{-q_-\beta_k}{z_{j_k}z+1}
\end{array}\right)\right] ,
\end{align}
where the coefficients $\beta_k=\beta_k(x,t)$, $\varsigma_k=\varsigma_k(x,t)$, $\alpha_s=\alpha_s(x,t)$ and $\tau_s=\tau_s(x,t)$  are determined later.

\begin{prop}\label{pp6}
There exists an unique solution of the RH problem \ref{RHP5}. Additionally, the solution $M^{(sol)}_\Lambda(z)$ is equivalent to the solution of the original RH problem \ref{RHP1} with modified scattering data $\left\lbrace  0,\left\lbrace \eta_n,\mathring{c}_n\right\rbrace_{n\in\Lambda}\right\rbrace$ under the condition of no reflection.
\begin{enumerate}[(I)]
\item For case 1: Let $\Lambda=\emptyset$. It's easy to check that
\begin{align}\label{msol1}
M^{(sol)}_\Lambda(z)=e^{i\nu_-(q^{(sol)}_\Lambda)\sigma_3}+\frac{i}{z}e^{i\nu_-(q^{(sol)}_\Lambda)\sigma_3}\sigma_3Q_-.
\end{align}
\item For case 2:  Let $\Lambda\neq\emptyset$. The the solution of $M^{(sol)}_{\Lambda}(z)$ is of the form \eqref{msol2}, where the  coefficients $\beta_k$, $\varsigma_k$, $\alpha_s$ and $\tau_s$  are determined by the following algebraic equation
\begin{align} c_{j_k}^{-1}T(z_{j_k})^{-2}&e^{-2i\theta(z_{j_k})t}\beta_k=\frac{i}{z_{j_k}}e^{i\nu_-(q^{(sol)}_\Lambda)}q_-+\sum_{h=1}^{n_2}
\left(\frac{\overline{\tau_h}}{z_{j_k}-\overline{w}_{i_h}}+\frac{\overline{\tau_h}}{z_{j_k}+\overline{w}_{i_h}} \right) \nonumber\\&+\sum_{\ell=1}^{n_1}\left( \frac{\overline{\varsigma_\ell}}{z_{j_k}-\bar{z}_{j_\ell}} +\frac{\overline{\varsigma_\ell}}{z_{j_k}+\overline{z}_{j_\ell}}-\frac{iq_-\varsigma_\ell}{z_{j_\ell}z_{j_k}-1}-
\frac{iq_-\varsigma_\ell}{z_{j_\ell}z_{j_k}+1}\right) , \\		 c_{j_k}^{-1}T(z_{j_k})^{-2}&e^{-2i\theta(z_{j_k})t}\varsigma_k=\frac{i}{z_{j_k}}e^{-i\nu_-(q^{(sol)}_\Lambda)}
\overline{q}_-+\sum_{h=1}^{n_2}\left(\frac{\overline{\alpha_h}}{z_{j_k}-\overline{w}_{i_h}}-\frac{\overline{\alpha_h}}
{z_{j_k}+\overline{w}_{i_h}} \right) \nonumber\\
&+\sum_{\ell=1}^{n_1}\left( \frac{\overline{\beta_\ell}}{z_{j_k}-\overline{z}_{j_\ell}} -\frac{\overline{\beta_\ell}}{z_{j_k}+\overline{z}_{j_\ell}}+\frac{iq_-\beta_\ell}{z_{j_\ell}z_{j_k}-1}-
\frac{iq_-\beta_\ell}{z_{j_\ell}z_{j_k}+1}\right) ,
\end{align}
and
\begin{align} c_{i_s+N_1}^{-1}T(w_{i_s})^{-2}&e^{-2i\theta(w_{i_s})t}\alpha_k=\frac{i}{w_{i_s}}e^{i\nu_-(q^{(sol)}_\Lambda)}q_-
+\sum_{h=1}^{n_2}\left(\frac{\overline{\tau_h}}{w_{i_s}-\overline{w}_{i_h}}+\frac{\overline{\tau_h}}{w_{i_s}+\overline{w}_{i_h}} \right) \nonumber\\
&+\sum_{\ell=1}^{n_1}\left( \frac{\overline{\varsigma_\ell}}{w_{i_s}-\overline{z}_{j_\ell}} +\frac{\overline{\varsigma_\ell}}{w_{i_s}+\overline{z}_{j_\ell}}-\frac{iq_-\varsigma_\ell}{z_{j_\ell}w_{i_s}-1}-
\frac{iq_-\varsigma_\ell}{z_{j_\ell}w_{i_s}+1}\right),\\	 c_{i_s+N_1}^{-1}T(w_{i_s})^{-2}&e^{-2i\theta(w_{i_s})t}\tau_k=\frac{i}{w_{i_s}}e^{-i\nu_-(q^{(sol)}_\Lambda)}\overline{q}_-
+\sum_{h=1}^{n_2}\left(\frac{\overline{\alpha_h}}{w_{i_s}-\overline{w}_{i_h}}-\frac{\overline{\alpha_h}}{w_{i_s}+\overline{w}_{i_h}} \right) \nonumber\\
&+\sum_{\ell=1}^{n_1}\left( \frac{\overline{\beta_\ell}}{w_{i_s}-\overline{z}_{j_\ell}} -\frac{\overline{\beta_\ell}}{w_{i_s}+\overline{z}_{j_\ell}}+\frac{iq_-\beta_l}{z_{j_\ell}w_{i_s}-1}-
\frac{iq_-\beta_\ell}{z_{j_\ell}
w_{i_s}+1}\right),
\end{align}
 for $k=1,\ldots,N_1$, $s=1,\ldots,N_2$, respectively.
\end{enumerate}
\end{prop}
\begin{cor}\label{cor-1}
For the reflection-less condition $\rho(z)=0$, the scattering matrix $S(z)$ and the jump matrix $V(z)$ are identity matrices, which further leads to the equality of boundary values $q_{\pm}$. For the modified scattering data $\{0,\{\eta,\mathring{c}_n\}_{n\in\Lambda}\}$, denote $q^{(sol)}_{\Lambda}(x,t)$ as the $\mathcal {N}(\Lambda)$-soliton solution of the mNLS \eqref{q1}, which can be expressed in the following form
\begin{align}
q^{(sol)}_{\Lambda}(x,t)=\lim_{z\rightarrow\infty}e^{i\nu_-(x,t;q^{(sol)}_{\Lambda})\sigma_3}(zM^{(sol)}_{\Lambda}(z))_{12},
\end{align}
from the formula \eqref{recover-q}.
Then for the case 1, it's easy to check that
\begin{align}
\left|q^{(sol)}_{\Lambda}(x,t)\right|=\lim_{z\rightarrow\infty}\left|(zM^{(sol)}_{\Lambda}(z))_{12}\right|=1.
\end{align}
Obviously, $\nu_-(x,t;q^{(sol)}_{\Lambda}(x,t))=0$ and $q^{(sol)}_{\Lambda}(x,t)=q_-$.

For case 2, similar to the calculation in the first case,  the following equality can be obtained
\begin{align}
q^{(sol)}_\Lambda(x,t)&=\lim_{z\to \infty}\left|\left(zM^{(sol)}_\Lambda \right)_{12}\right|\nonumber\\ &=\left|ie^{i\nu_-(x,t;q^{(sol)}_\Lambda)}q_-+2\sum_{s=1}^{n_2}\overline{\tau}_k+2\sum_{k=1}^{n_1}
(\overline{\varsigma}_k-iq_-\varsigma_k)\right|.
\end{align}
Finally, the $\mathcal {N}(\Lambda)$-soliton solutions $q^{(sol)}_\Lambda(x,t)$ of  the mNLS \eqref{q1} in this case with modified scattering data can be derived
\begin{align}
q^{(sol)}_\Lambda(x,t)=e^{2i\nu_-(x,t;q^{(sol)}_\Lambda)}\left(ie^{i\nu_-(x,t;q^{(sol)}_\Lambda)}q_-+
2\sum_{s=1}^{n_2}\overline{\tau}_k+2\sum_{k=1}^{n_1}(\overline{\varsigma}_k-iq_-\varsigma_k) \right),
\end{align}
where $\nu_-(x,t;q^{(sol)}_\Lambda)=\frac{1}{2\alpha}\int_{-\infty}^x (1-|q^{(sol)}_\Lambda(y,t)|^2)dy$.
\end{cor}

\subsection{The small norm theory for the RH problem of error function}\quad
On the other hand, it is noted that $M^{(2)}(z)$ is the solution of the mixed RH problem \ref{RHP3} without $\overline{\partial}$-derivative. Since its jump matrix $V^{(2)}(z)$ converges uniformly to the identity matrix, the contribution of the contour $\Sigma^{(2)}$ to the long-time asymptotic behavior of the solution can be ignored. Then we consider the decomposition of $M^{(2)}(z)$, that is, the expression \eqref{ML}. For the error function $M^{(err)}(z)$, the following problem is satisfied.
\begin{RHP}\label{RHP6}
Find a matrix valued function $M^{(err)}(z)$ satisfies
\begin{enumerate}[(1)]
\item The error function $M^{(err)}(z)$  is analytic in $\mathbb{C}\setminus\Sigma^{(2)}$.
\item For the asymptotic behavior:
\begin{align*}
M^{(err)}(z)=I+\mathcal {O}(z^{-1}), ~~~z\rightarrow\infty.
\end{align*}
\item The jumping relationship:
\begin{align}
M^{(err)}_{+}(z)=M^{(err)}_{-}(z)V^{(err)}(z),
\end{align}
where $V^{(err)}(z)=M^{(sol)}_\Lambda(z)V^{(2)}(z)M^{(sol)}_\Lambda(z)^{-1}$, and for $z\in\Sigma^{(2)}$ the error function $M^{(err)}(z)$ is of continuous boundary values $M^{(err)}_{\pm}(z)$.
\end{enumerate}
\end{RHP}
Since $M^{(sol)}(z)$ and $M_\Lambda^{(sol)}(z)$ have the same jump condition $V^{(2)}(z)$, it is easy to verify that $M^{(err)}(z)$ has no poles. Additionally, the jump matrix $V^{(err)}(z)$ generated by the error function $M^{(err)}(z)$ on contour $\Sigma^{(2)}$ satisfies the same estimation as $V^{(2)}(z)$, i.e.,  $\left|\left|V^{(err)}(z)-I\right|\right|_{L^{p}(\Sigma^{(2)})}=\mathcal {O}(e^{-2\varrho_0t})$ for $1\leq p\leq \infty$, which indicates that the jump matrix $V^{(err)}(z)$ uniformly converges to the identity matrix. Then, it follows from the small norm theory of RH problem \cite{small-1,small-2} that the solution of RH problem \ref{RHP6} can be uniquely expressed as
\begin{align}\label{Ez}
M^{(err)}(z)=I+\frac{1}{2\pi i}\int_{\Sigma^{(2)}}\dfrac{\left( I+\eta(s)\right) (V^{(err)}(s)-I)}{s-z}ds,
\end{align}
where $\eta\in L^{(2)}(\Sigma^{(2)})$ is the unique solution of the operator equation
\begin{align}\label{OE}
(1-C_{V^{(err)}})\eta=C_{V^{(err)}}I,
\end{align}
where the operator $C_{V^{(err)}}:$ $L^2(\Sigma^{(2)})\to L^2(\Sigma^{(2)})$ denotes the Cauchy projection operator
\begin{align*}
C_{V^{(err)}}[f](z)=C_{-}[f(V^{(err)}-I)]=\lim_{k\rightarrow z}\int_{\Sigma^{(2)}}\frac{f(s)(V^{(err)(s)}-I)}{s-k}ds.
\end{align*}
The definition of the limit $C_-$ here is similar to the previous one, that is, the limit is obtained from the right side of the oriented contour $\Sigma^{(2)}$ with a non-tangent. Additionally, the operator equation \eqref{OE} is meaningful, namely, the existence and uniqueness of solution $\eta$ can be guaranteed by the boundedness of the Cauchy projection operator $C_-:$ $\left|\left|C_{V^{(err)}}\right|\right|_{L^2(\Sigma^{(2)})\to L^2(\Sigma^{(2)})}=\mathcal {O}(e^{-2\varrho_0t})$. So far, the existence and uniqueness of the error function $M^{(err)}(z)$ are established. In order to obtain the long-time asymptotic behavior of the mNLS equation \eqref{q1}, we need to further consider the asymptotic behavior of the error function shown in the next proposition.
\begin{prop}\label{pp3}
For the error function $M^{(err)}(z)$, the following estimate  is satisfied
\begin{align}\label{Merr}
|M^{(err)}(z)-I|\lesssim\mathcal{O}(e^{- 2\varrho_0t}).
\end{align}
For $z\rightarrow\infty$, $M^{(err)}(z)$ has the following expansion
\begin{align}
M^{(err)}(z)=I+M_1^{(err)}(z)z^{-1}+\mathcal {O}(z^{-2}),
\end{align}
where
\begin{align}
M_1^{(err)}(z)=-\frac{1}{2\pi i}\int_{\Sigma^{(2)}}\left( I+\eta(s)\right) (V^{(err)}(s)-I)ds,
\end{align}
and satisfies the following estimate
\begin{align}
M_1^{(err)}(z)
\lesssim\mathcal{O}(e^{- 2\varrho_0t}).\label{E1t}
\end{align}
\end{prop}
\begin{proof}
It follows the operator equation \eqref{OE} that
\begin{align}
M^{(err)}(z)&=I+\frac{1}{2\pi i}\int_{\Sigma^{(2)}}\dfrac{\left( I+\eta(s)\right) (V^{(err)}(s)-I)}{s-z}ds\nonumber\\
&=I+C_{V^{(err)}}(I+\eta(z))=I+C_{V^{(err)}}I+C_{V^{(err)}}\eta(z)\nonumber\\
&=I+(1-C_{V^{(err)}})\eta(z)+C_{V^{(err)}}\eta(z),
\end{align}
from which the estimation can be derived
\begin{align}
|M^{(err)}(z)-I|\leq|(1-C_{V^{(err)}})\eta(z)|+|C_{V^{(err)}}\eta(z)|\lesssim\mathcal{O}(e^{- 2\varrho_0t}).
\end{align}
The proof of the equality \eqref{Merr} is completed, then we consider  equality \eqref{E1t}. Taking the limit of the expression \eqref{Ez} for $z\rightarrow\infty$ yields
\begin{align}
M^{(err)}(z)&=I+\frac{1}{2z\pi i}\int_{\Sigma^{(2)}}\dfrac{\left( I+\eta(s)\right) (V^{(err)}(s)-I)}{\frac{s}{z}-1}ds\nonumber\\
&=I-\frac{1}{z}\frac{1}{2\pi i}\int_{\Sigma^{(2)}}\left( I+\eta(s)\right) (V^{(err)}(s)-I)ds,
\end{align}
which further indicates that
\begin{align}
|M_1^{(err)}(z)|\lesssim \parallel V^{(err)}-I \parallel_{L^1(\Sigma^{(2)})}+\parallel \eta \parallel_{L^2(\Sigma^{(2)})}\parallel V^{(err)}-I \parallel_{L^2(\Sigma^{(2)})}\lesssim\mathcal{O}(e^{- 2\varrho_0t}).
\end{align}
\end{proof}
\subsection{The pure Dbar-problem}\quad
The mixed RH problem \ref{RHP3} satisfied by the matrix $M^{(2)}(z)$ includes two cases, one is the case of soliton solution corresponding to $\partial$-derivative $W(z)\equiv0$, which has been analyzed in detail in the previous work. Now what we need to do is to remove the soliton component from the mixed RH problem and get the pure Dbar problem which can be analyzed by using the Dbar technique in detail. Specifically, a new matrix valued function $M^{(3)}(z)$ is defined with the help of  the function $M^{(sol)}(z)$ satisfying the  RH problem \ref{RHP4}
\begin{align}\label{M3}
M^{(3)}(z)=M^{(2)}(z)\left(M^{(sol)}(z)\right)^{-1},
\end{align}
which admits the following pure $\overline{\partial}$ problem.
\begin{RHP}\label{RHP7}
Find a matrix-valued function $M^{(3)}(z)$ is of the following properties:
\begin{enumerate}[(1)]
\item $M^{(3)}(z)$ is continuous, and  has sectionally continuous first partial derivatives  in $\mathbb{C}$.
\item The asymptotic behavior: $M^{(3)}(z)=I+\mathcal {O}(z^{-1})$ as $z\rightarrow\infty$.
\item $\overline{\partial}$-derivative: For $z\in\mathbb{C}$, one has
\begin{align}\label{W3}
\overline{\partial}M^{(3)}(z)=M^{(3)}(z)W^{(3)}(z),
\end{align}
where $W^{(3)}(z)=M^{(sol)}(z)W(z)M^{(sol)}(z)^{-1}$ with the matrix $W(z)$ defined by \eqref{Wz}.
\end{enumerate}
\end{RHP}
\begin{proof}
Since the matrix-valued functions $M^{(2)}(z)$ and $M^{(sol)}(z)$ satisfying RH problem \ref{RHP3} and \ref{RHP4} respectively have the same jump condition $V^{(2)}(z)$, it follows from expression \eqref{M3} that $M^{(3)}(z)$ has no jump on disks $|z-\eta_n|=\epsilon_0$ and $|z-\overline{\eta}_n|=\epsilon_0$. Additionally, the asymptotic behavior and $\overline{\partial}$-derivative of the function $M^{(3)}(z)$ inherit immediately from the related properties of $M^{(2)}(z)$ and $M^{(sol)}(z)$. Next we will show that $M^{(3)}(z)$ has no singularity at the origin. Observing that $M^{(sol)}(z)=(1+z^{-2})^{-1}\sigma_2M^{(sol)}(z)^{T}\sigma_2$ from the fact that for any $2\times2$ invertible matrix $A$
satisfying $A^{-1}=(\det A)^{-1}\sigma_2A^{T}\sigma_2$, we further have
\begin{align}
\lim_{z\to 0}M^{(3)}(z)=&\lim_{z\to 0}\dfrac{(M^{(2)}(z))\sigma_2(M^{(sol)}(z)^T)\sigma_2}{1+z^{-2}}
=\lim_{z\to 0}\dfrac{(zM^{(2)}(z))\sigma_2(zM^{(sol)}(z)^T)\sigma_2}{1+z^2}\nonumber\\
=&ie^{i\nu_-(x,t;q)\sigma_3}\sigma_3Q_-\sigma_2(ie^{i\nu_-(x,t;q)\sigma_3}\sigma_3Q_-)^T\sigma_2=I.
\end{align}
Since $M^{(2)}(z)$ and $M^{(sol)}(z)$ have the same jump matrix, then one has
\begin{align*}
M_-^{(3)}(z)^{-1}M_+^{(3)}(z)&=M_-^{(2)}(z)^{-1}M_-^{(sol)}(z)M_+^{(sol)}(z)^{-1}M_+^{(2)}(z)\\
&=M_-^{(2)}(z)^{-1}V^{(2)}(z)^{-1}M_+^{(2)}(z)=I,
\end{align*}
which means that $M^{(3)}(z)$ is regular at the origin.  Meanwhile, $M^{(3)}(z)$ has no poles for $\xi\in\{\eta_n,\overline{\eta}_n\}_{n\in\Lambda}$. Let $A_\xi$ be the nilpotent matrix appearing in the residue conditions of the RH problem \ref{RHP3} and \ref{RHP4}, and then we have a local expansion in the neighborhood of $\xi$
\begin{align}
&M^{(2)}(z)=\alpha(\xi) \left[ \dfrac{A_\xi}{z-\xi}+I\right] +\mathcal{O}(z-\xi),\nonumber\\
&	M^{(sol)}(z)=\beta(\xi) \left[ \dfrac{A_\xi}{z-\xi}+I\right] +\mathcal{O}(z-\xi),\nonumber
\end{align}
where $\alpha(\xi)$ and $\beta(\xi)$ are constant coefficients of the corresponding expansion respectively. Taking the product yields that $M^{(3)}(z)=\mathcal{O}(1)$, which indicates that the poles of $M^{(3)}(z)$  are removable and $M^{(3)}(z)$ is boundary locally. The proof of the third property follows the following process
\begin{align*}
\overline{\partial}M^{(3)}(z)&=\overline{\partial}(M^{(2)}(z)M^{(sol)}(z)^{-1})=(\overline{\partial}M^{(2)}(z))
M^{(sol)}(z)^{-1}\\&=M^{(2)}(z)W(z)M^{(sol)}(z)^{-1}=M^{(2)}(z)M^{(sol)}(z)^{-1}M^{(sol)}(z)W(z)M^{(sol)}(z)^{-1}\\
&=M^{(3)}(z)M^{(sol)}(z)W(z)M^{(sol)}(z)^{-1}\triangleq M^{(3)}(z)W^{(3)}(z).
\end{align*}

Finally, the singular behavior of $M^{(3)}(z)$ at $z=\pm i$ should be checked by the reason that  $\det M^{(sol)}(z)=1+z^{-2}$. The symmetries shown in RH problem \ref{RHP3} and \ref{RHP4} are applied to  $M^{(2)}(z) $ and $M^{(sol)}(z)$ to obtain the following local expansions in a neighborhood of the singularities $z=\pm i$
\begin{align}
&M^{(2)}(z)=\left(\begin{array}{cc}
 		\vartheta & \pm q_-\vartheta\\
 		\pm \bar{q}_-\overline{\vartheta}& \overline{\vartheta}
 	\end{array}\right)+\mathcal{O}(z\mp i),\\
&M^{(sol)}(z)=\dfrac{\pm i}{2(z\mp i)}\left(\begin{array}{cc}
 	\overline{\varpi}& \mp q_-\varpi\\
 	\mp \overline{q}_-\overline{\varpi}& \varpi
 \end{array}\right)+\mathcal{O}(1),
 \end{align}
where $\vartheta$ and $\varpi$ are constants. By directly making product, it can be immediately verified that the singularity of $M^{(3)}(z)$ at singularities $z=\pm i$ is disappeared.
\end{proof}

In order to establish the long-time asymptotic solution of the mNLS equation \eqref{q1}, the asymptotic behavior of $M^{(3)}(z)$ needs to be further considered.

%Next, we will combine the known Dbar-technology to analyze the pure Dbar RH problem \ref{RHP6}. First, it is noted that $M^{(3)}(z)$ satisfying the RH problem \ref{RHP6} can be equivalent to the following integral equation
%\begin{equation}
%M^{(3)}(z)=I+\frac{1}{\pi}\int_\mathbb{C}\dfrac{M^{(3)}(s)W^{(3)}(s)}{z-s}dm(s),
%\end{equation}
\begin{prop}\label{pp5}
As $z\rightarrow\infty$, $M^{(3)}(z)$ is of the following asymptotic expansion
\begin{align}
M^{(3)}(z)=I+\frac{M_1^{(3)}(x,t)}{z}+\mathcal {O}(z^{-2}),
\end{align}
where $M_1^{(3)}(x,t)$ is independent of the variable $z$, there then exists a constant $t_1$ such that
\begin{align}
|M_1^{(3)}(x,t)|\lesssim t^{-3/4}, ~~\text{as}~~t>t_1.
\end{align}
\end{prop}

Before proving the proposition \ref{pp5}, we first consider the left Cauchy-Green operator shown in the lemma.
\begin{lem}\label{lem4}
Consider the left Cauchy-Green integral operator $C_{z}$:
\begin{align}\label{Cauchy}
fC_z(z)=\frac{1}{\pi}\int_C\dfrac{f(s)W^{(3)}(s)}{z-s}dA(s).
\end{align}
The $A(s)$ denotes the Lebesgue measure on the $\mathbb{C}$. Then for $f$: $L^{\infty}(\mathbb{C})\rightarrow L^{\infty}(\mathbb{C})\cap C^{0}(\mathbb{C})$, the integral operator $C_{z}$ satisfies the estimate as $t\rightarrow\infty$
\begin{align}
\left|\left|C_{z}\right|\right|_{L^{\infty}(\mathbb{C})\rightarrow L^{\infty}(\mathbb{C})}\lesssim t^{-1/2}.
\end{align}
\end{lem}
\begin{proof}
For any $f\in L^{\infty}(\mathbb{C})$, it follows from expression \eqref{Cauchy} that
\begin{align}
\parallel fC_z \parallel_{L^\infty}&\leq \parallel f \parallel_{L^\infty}\frac{1}{\pi}\int_C\dfrac{|W^{(3)}(s)|}{|z-s|}dA(s),\nonumber
\end{align}
where $W^{(3)}(z)$ is defined in \eqref{W3}. According to  the Dbar-derivative $W(z)$  is equal to zero outside region $\Omega$, we only need to estimate the following expression now
\begin{align*}
\frac{1}{\pi}\int_\Omega\dfrac{|W^{(3)}(s)|}{|z-s|}dA(s).
\end{align*}
From the propositions \ref{pp3} and \ref{pp6}, the matrix norm of $|M^{(sol)}(z)|\leq\sqrt{1+z^{-2}}$. Hereafter, we further have
\begin{align}\label{guji}
\frac{1}{\pi}\int_\Omega\dfrac{|W^{(3)}(s)|}{|z-s|}dA(s)\lesssim
\frac{1}{\pi}\int_\Omega\dfrac{|W(s)|}{|z-s|}\frac{1+|s|^{-2}}{|1+s^{-2}|}dA(s).
\end{align}

For $j=1,4,5,8$, the norm of $M^{(sol)}(z)$ are bounded in the regions $\Omega_j$.  Since there are singularities $z=\pm i$ caused by the determinant $\det M^{(sol)}(z)$ in regions $\Omega_j$ $(j=2,3,6,7)$, the estimation at this time needs to be handled carefully. Taking the region $\Omega_2$ as an example, other regions can be proved similarly. For convenience, the  region $\Omega_2$ is divided into three parts $I_{i}$ $(i=1,2,3)$
\begin{align}	
\begin{split}
I_1=\mathbb{D}&(0,1-\upsilon_0)\cap\Omega_2,\hspace{0.3cm}I_3=\Omega_2\setminus\mathbb{D}(0,1+\upsilon_0),\\
&I_2=\mathbb{D}(0,1+\upsilon_0)\setminus\mathbb{D}
(0,1-\upsilon_0)\cap \Omega_2,
\end{split}
\end{align}
which are shown in Fig. 6.\\

\centerline{
		\begin{tikzpicture}[node distance=2cm]
\draw[gray, fill=gray!40] (0,0)--(2,3)--(0,3)--(0,0);
\draw[->][thick](-2,0)--(3,0);
\draw[->][thick](0,0)--(0,3);
\draw [thick](0,0)--(2,3);
\draw[thick][dashed] (0.8,0) arc (0:90:0.8);
\draw[thick][dashed] (2.1,0) arc (0:90:2.1);
\draw[fill] (0.1,0.3)node[above left]{$\small{I_1}$};
\draw[fill] (0.1,1.2)node[above left]{$\small{I_2}$};
\draw[fill] (0.1,2.2)node[above left]{$\small{I_3}$};
\draw[fill] (0,0)node[below]{$0$};
\draw[fill] (0.8,0)node[below]{$\upsilon_0$};
\draw[fill] (2.1,0)node[below]{$1+\upsilon_0$};
\end{tikzpicture}}
\centerline{\noindent {\small \textbf{Figure 6.} The three sub-region of $\Omega_2$. }}

Then the integral equation \eqref{guji} can be divided into three parts, we now consider \eqref{guji} in the region $I_3$.
\begin{align}
J_3=\int_{I_3}\dfrac{|W(s)|}{|z-s|}\frac{1+|s|^{-2}}{|1+s^{-2}|}dA(s). 	
\end{align}

For any $z=x+iy\in I_3$ with $x,y\in\mathbb{R}$, and let $s=u+iv\in I_3$. It is easy to verify that the second integrand in the integral equation \eqref{guji} is bounded from $|s|>1+\upsilon_0$. It follows from \eqref{DBARR2} that
\begin{align}
J_3\lesssim\int_{\Omega_2}\dfrac{|W(s)|}{|z-s|}dm(s)=\int_{\Omega_2}\dfrac{|\overline{\partial}R_2 (s)e^{-2it\theta}|}{|z-s|}dA(s) \label{I_3}.
\end{align}
Additionally, by the fact that $|e^{-2it\theta}|=e^{-Re(2it\theta)}\leq e^{-cuvt}\leq e^{-cut}$ in the region $I_3$ from the proposition \ref{pp2}, one has
\begin{align}
J_3\lesssim&\int_{0}^{+\infty}\int_{\frac{u}{\tan \psi}}^{+\infty}\dfrac{|\mathcal {H}_2' (i|s|)|e^{-cut}}{|z-s|}dvdu+\int_{0}^{+\infty}\int_{\frac{u}{\tan \psi}}^{+\infty}\dfrac{|s|^{-1/2}e^{-cut}}{|z-s|}dvdu\nonumber\\
&+\int_{0}^{+\infty}\int_{\frac{u}{\tan \psi}}^{+\infty}\dfrac{|\overline{\partial}\mathcal {X}_1(|s|)|e^{-cut}}{|z-s|}dvdu\triangleq J_3^{(1)}+J_3^{(2)}+J_3^{(3)},\nonumber
\end{align}
where $\left|\left||s-z|^{-1}\right|\right|_{L^{q}(\mathbb{R}^{+})}\leq|u-x|^{-1/p}$ for $1\leq q\leq+\infty$ with $\frac{1}{p}+\frac{1}{q}=1$ is used to estimate the inequality.

For the $J_3^{(1)}$, by using Cauchy-Schwarz inequality, we have
\begin{align}
J_3^{(1)}&=\int_{0}^{+\infty}\int_{\frac{u}{\tan \psi}}^{+\infty}\dfrac{|\mathcal {H}_2' (i|s|)|e^{-cut}}{|z-s|}dvdu\nonumber\\
&\leq \int_{0}^{+\infty}\parallel\widetilde{\rho}'\parallel_{L^2(i\mathbb{R})}\parallel |s-z|^{-1}\parallel_{L^2(\mathbb{R}^+)}e^{-cut}du\nonumber\\
&\leq \int_{0}^{+\infty}e^{-cut}|u-x|^{-\frac{1}{2}}du\lesssim t^{-\frac{1}{2}}.
\end{align}
Similar methods can verify the estimation of $J_3^{(3)}$, and now only $J_3^{(2)}$ needs to be estimated. For $p>2$ such that $\frac{1}{p}+\frac{1}{q}=1$, the following inequality holds
\begin{align*}
\left|\left||s|^{-1/2}\right|\right|_{L^{p}_v(\frac{u}{\tan \psi},\infty)}&=\left(\int_{\frac{u}{\tan \psi}}
^{\infty}\frac{1}{|u+iv|^{\frac{p}{2}}}dv\right)^{1/p}\\
%&\leq\left(\int_{\mathbb{R}}\frac{1}{(u^2+v^2)^{p/4}}dv\right)^{1/p}=\left(\int_{\mathbb{R}}\frac{u}{u^{p/2}
%(1+x^2)^{p/4}}dx\right)^{1/p}\\
&\lesssim u^{\frac{1}{p}-\frac{1}{2}}\left(\int_{\mathbb{R}}\frac{1}{(1+x^2)^{p/4}}\right)^{1/p}\lesssim u^{\frac{1}{p}-\frac{1}{2}}.
\end{align*}
It follows from the above inequality that the estimation of $J_3^{(2)}$ such that
\begin{align*}
J_3^{(2)}&=\int_{0}^{+\infty}\int_{\frac{u}{\tan \psi}}^{+\infty}\dfrac{|s|^{-1/2}e^{-cut}}{|z-s|}dvdu\\
&\lesssim \int_{0}^{\infty}\left|\left||s|^{-1/2}\right|\right|_{L^{p}_v(\frac{u}{\tan \psi},\infty)}
\left|\left||z-s|^{-1}\right|\right|_{L^{q}(\mathbb{R}^+)}e^{-cut}du\\
&\lesssim \int_{0}^{\infty}u^{\frac{1}{p}-\frac{1}{2}}|u-x|^{-1/p} e^{-cut}du\lesssim t^{-\frac{1}{2}}.
\end{align*}
Therefore we prove  $J_{3}\lesssim t^{-\frac{1}{2}}$. Now considering $J_{2}$, we notice that the singularity of $J_{2}$ at  $i$ is eliminated by expression \eqref{Ri}. Additionally, $0<\upsilon_0<1$ and $(1-\upsilon_0)\cos \psi>\frac{1}{2}$.
\begin{align}
J_2&\lesssim\int_{0}^2\int_{1/2}^{2}\dfrac{e^{-cut}}{|z-s|}\frac{1+|s|^2}{|s+i|}dvdu\lesssim \int_{0}^2\int_{1/2}^{2}\dfrac{e^{-2cut}}{|z-s|}dvdu\nonumber\\
&\lesssim \int_{0}^2 |u-x|^{-1/2}e^{-cut}du\lesssim t^{-1/2}.
\end{align}
Finally, for $J_1$, similar to the prove of $J_3$, one has
\begin{align}
J_1\lesssim \int_{0}^{1-\upsilon_0}\int_{u}^{1-\upsilon_0}\left( |\mathcal {H}_2' (i|s|)|+||s||^{-1/2}+|\overline{\partial}\mathcal {X}_1(|s|)|\right) \dfrac{e^{-cut}}{|z-s|}dvdu\lesssim t^{-1/2}.
\end{align}
So far, the proof of lemma \ref{lem4} is completed.
\end{proof}

On the basis of lemma \ref{lem4}, that is, the Cauchy-Green operator $C_z$ with good properties, we turn to prove proposition \ref{pp5} exhibited in the following contents.
\begin{proof}
An equivalent integral solution of $M^{(3)}(z)$ satisfying RH problem \ref{RHP7} such that
\begin{align}
M^{(3)}(z)=I+\frac{1}{\pi}\int_\mathbb{C}\dfrac{M^{(3)}(s)W^{(3)}(s)}{z-s}dA(s),
\end{align}
which can be written as the operator form
\begin{align}
M^{(3)}(z)=I\cdot\left(I-C_z \right)^{-1}.\label{deM3}
\end{align}
By using the operator $C_z$, the existence of operator $\left(I-C_z \right)^{-1}$ be guaranteed by the lemma \ref{lem4}. Additionally, from the expression \eqref{deM3} and the lemma \ref{lem4} as $t\rightarrow\infty$,  $||M^{(3)}(z)||_{L^{\infty}}\leq1$. Similar to prove the lemma \ref{lem4}, we only need to consider the estimation on region $\Omega_2$. The region $\Omega_2$ is similarly divided into three parts.
\begin{align}	
\frac{1}{\pi}\int_{\mathbb{C}}M^{(3)}(s)W^{(3)}(s)dA(s)&\lesssim \int_{\Omega_2}\left|W(s)
e^{-2it\theta}\right|\frac{1+|s|^{-2}}{|1+s^{-2}|}dA(s)\nonumber\\
&=\left(\int_{I_1}+\int_{I_2}+\int_{I_3}\right)
\left|W(s)e^{-2it\theta}\right|\frac{1+|s|^{-2}}{|1+s^{-2}|}dA(s)\nonumber\\
&\triangleq J_{4}+J_{5}+J_{6}.
\end{align}
For $J_6$, it follows by the fact $\frac{1+|s|^{-2}}{|1+s^{-2}|}<\infty$ that
\begin{align}
J_6\lesssim&\int_{0}^{+\infty}\int_{\frac{u}{\tan \psi}}^{+\infty}|\mathcal {H}_2'(i|s|)|e^{-2cuvt}dvdu+\int_{0}^{+\infty}\int_{\frac{u}{\tan \psi}}^{+\infty}|r|^{-\frac{1}{2}}e^{-cuvt}dvdu\nonumber\\
&+\int_{0}^{+\infty}\int_{\frac{u}{\tan \psi}}^{+\infty}|\overline{\partial}\mathcal {X}_1(|s|)|e^{-cuvt}dvdu
\triangleq J_{6}^{(1)}+J_{6}^{(2)}+J_{6}^{(3)}.\label{I6}
\end{align}
It is easy to check that $\left(\int_{\frac{u}{\tan \psi}}^{+\infty} e^{-2cuvtq}dv\right)^{\frac{1}{q}}
\lesssim e^{-u^2t}(ut)^{-\frac{1}{q}}$. Then for the integral $J_{6}^{(1)}$, one has
\begin{align}
J_{6}^{(1)}
\lesssim t^{-\frac{1}{2}}\int_{0}^{+\infty}\parallel\widetilde{\rho}'\parallel_{L^2(i\mathbb{R})} u^{-\frac{1}{q}}e^{-u^2t}du\lesssim t^{-\frac{3}{4}}.\nonumber
\end{align}
Similar methods can be used to estimate $J_{6}^{(3)}$. Finally,  with the help of the Cauchy-Schwarz inequality, there are
\begin{align}
J_{6}^{(2)}\lesssim t^{-\frac{1}{q}}\int_{0}^{+\infty}u^{\frac{2}{p}-\frac{3}{2}}e^{-u^2t}du\lesssim t^{-\frac{3}{4}},
\end{align}
for $4>p>2$ such that $\frac{1}{q}+\frac{1}{p}=1$. The method of proving $J_{6}$ can be applied to proving $J_{4}$. Now considering  $J_{5}$, from the estimation \eqref{Ri} we have
\begin{align}
J_{5}&\lesssim\int_{I_2}\dfrac{e^{-cut}}{|z-s|}\frac{1+|s|^2}{|s+i|}dA(s)\lesssim \int_{1/2}^2\int_{u}^{2}e^{-cuvt}dvdu\nonumber\\
&=\int_{1/2}^2(cut)^{-1}\left(e^{-cu^2t}-e^{-cut} \right)du \lesssim t^{-1}.
\end{align}
\end{proof}

\subsection{The asymptotic solution}\quad
Here we will reconstruct the long-time asymptotic solution of the mNLS equation \eqref{q1} under nonzero boundary conditions \eqref{q2} based on the previous analysis. The inverse transformation of \eqref{M1}, \eqref{M2}, \eqref{M3} and \eqref{ML} establishes the following relationship
\begin{align}
M(z)=T(\infty)^{\sigma_3}M^{(3)}(z)M^{(err)}(z)M^{(sol)}_\Lambda(z)R^{(2)}(z)^{-1}T(z)^{-\sigma_3}.
\end{align}
Taking $z\rightarrow\infty$ out of $\overline{\Omega}$, the solution of the mNLS equation \eqref{q1} can be derived
\begin{align}
M(z)=&T(\infty)^{\sigma_3}\left(I+ M^{(3)}_1(z)z^{-1}\right)M^{(err)}(z)M^{(sol)}_\Lambda(z)T(\infty)^{-\sigma_3}\times\nonumber\\
	&\left(1+z^{-1}\frac{1}{2\pi i} \int _{\mathbb{R}}\log (1-\rho(s)\widetilde{\rho}(s))ds\right)^{-\sigma_3} + \mathcal{O}(z^{-2}),
\end{align}
by using the lemma \ref{lem3}, propositions \ref{pp3} and \ref{pp5}. Additionally, the mNLS equation \eqref{q1} possesses long time asymptotic behavior
\begin{align}
M(z)=&T(\infty)^{\sigma_3} M^{(sol)}_\Lambda(z)T(\infty)^{-\sigma_3}\left(1+\frac{1}{2z\pi i}\int _{\mathbb{R}}\log (1-\rho(s)\widetilde{\rho}(s))ds\right)^{-\sigma_3}\nonumber\\
& + \mathcal{O}(z^{-2})+\mathcal{O}(t^{-3/4}).\nonumber
\end{align}
It follows from formula \eqref{recover-q} that
\begin{align}
q(x,t)&=exp\left(\frac{i}{\alpha}\int_{\pm\infty}^{x}\left(1-|q_\Lambda^{(sol)}(x,t)|^{2}\right)dy\right)
\lim_{z\rightarrow\infty}(zM(z))_{12}\nonumber\\
&=exp\left( \frac{i}{\alpha}\int_{\pm\infty}^x\left(1-|q_\Lambda^{(sol)}(x,t)|^{2}\right)dy \right) T(\infty)^{2}q^{(sol)}_\Lambda(x,t)+\mathcal{O}(t^{-3/4}),
\end{align}
where $q^{(sol)}_\Lambda(x,t)$ is defined by the corollary \ref{cor-1}. Finally, we have completed the proof of the main theorem of this work.

The asymptotic stability of soliton solutions of mNLS equation \eqref{q1} can be expressed as the following theorem.
\begin{thm}\label{first}
Consider the scattering data  $\left\lbrace  \rho(z),\left\lbrace \eta_n,C_n\right\rbrace^{4N_1+2N_2}_{n=1}\right\rbrace$ generated by the initial data $q_0(x)-q_{\pm}\in H^{1,1}(\mathbb{R})$. For the soliton region of space-time $\xi=\frac{x}{t}$ with $5<\xi<7$, let $q^{(sol)}_\Lambda(x,t)$ be the $\mathcal{N}(\Lambda)$-soliton solution corresponding to the modified scattering data
$\left\lbrace  0,\left\lbrace \eta_n,\mathring{c}_n\right\rbrace_{n\in\Lambda}\right\rbrace$ shown in corollary \ref{cor-1}.
Then for a large constant  $t_1$, as $t_1<t\to\infty$
\begin{align}
q(x,t)=exp\left\lbrace \frac{i}{\alpha}\int_{\pm\infty}^x(1-|q^{(sol)}_\Lambda(x,t)|^2)dy \right\rbrace T(\infty)^{2}q^{(sol)}_\Lambda(x,t)+\mathcal{O}(t^{-3/4}),\label{resultu}
\end{align}
where $\Lambda$, $q^{(sol)}_\Lambda(x,t)$ and $T(z)$ are shown in  \eqref{devide}, corollary \ref{cor-1} and  lemma \ref{lem3}  respectively.
\end{thm}
\begin{rem}
Actually, the constraint on $\xi=x/t\in(5,7)$  in the theorem \eqref{first} is only used to limit the length of this work. This is a critical region. The other regions shown in Figure 3 can also be discussed similarly. That is, in combination with \cite{fNLS}, to deal with the situation with stable phase points, the parabolic cylinder function can be used as usual to match, and the small neighborhood of the phase point provides the leading term for the long-time asymptotic behavior of the solution.
\end{rem}

\section*{Acknowledgements}\quad
This work was supported by the National Natural Science Foundation of China under Grant No. 11975306, the Natural Science Foundation of Jiangsu Province under Grant No. BK20181351, the Six Talent Peaks Project in Jiangsu Province under Grant No. JY-059, and the Fundamental Research Fund for the Central Universities under the Grant Nos. 2019ZDPY07 and 2019QNA35.

\renewcommand{\baselinestretch}{1.2}

\end{document}